\documentclass[oneside,reqno,11pt]{amsart}
\usepackage{amssymb,amsthm,amstext,amsfonts}
\usepackage[normalem]{ulem}
\usepackage{amscd,amstext,amsthm,amsfonts,latexsym}
\usepackage[shortlabels]{enumitem}
\usepackage[colorlinks=true,linkcolor=blue, urlcolor=black, citecolor=blue]{hyperref}
\hypersetup{
 citecolor=blue}
\theoremstyle{Theorem A}
\theoremstyle{Theorem B}
\theoremstyle{Theorem C}
\theoremstyle{Theorem D}
\theoremstyle{Theorem E}

\let\oldH\H

\allowdisplaybreaks

\makeatletter

\theoremstyle{definition}
\newtheorem{maintheorem}{Theorem}

\numberwithin{equation}{section}
\numberwithin{figure}{section}
\theoremstyle{plain}
\newtheorem*{cor*}{\protect\corollaryname}
\theoremstyle{plain}
\newtheorem{thm}{\protect\theoremname}[section]
\theoremstyle{definition}
\newtheorem{defn}[thm]{\protect\definitionname}
\theoremstyle{question}

\theoremstyle{remark}

\theoremstyle{plain}
\newtheorem{prop}[thm]{\protect\propositionname}
\theoremstyle{plain}
\newtheorem{lem}[thm]{\protect\lemmaname}
\theoremstyle{plain}

\usepackage{geometry}
\usepackage{amsmath}

\numberwithin{equation}{section}
\numberwithin{figure}{section}
\usepackage{enumitem}		
 \let\footnote=\endnote
\@ifundefined{lettrine}{\usepackage{lettrine}}{}

\theoremstyle{definition}

\def\L{\Lambda}

\def\D{\mathbb{D}}
\def\R{\mathbb{R}}
\def\C{\mathbb{C}}
\def\H{\mathbb{H}}
\def\N{\mathbb{N}}

\def\ep{\varepsilon}

\usepackage{float}

\keywords{}

\subjclass[2000]{}

\def\Si{\Sigma}

\def\D{\Delta}
\def\R{\mathbb{R}}

\def\C{\mathbb{C}}
\def\Grass{\text{Grass}}

\def\ep{\varepsilon}

\def\diml{\text{dim}_{\text{L}}}

\def\A{\mathcal{A}}
\def\B{\mathcal{B}}
\def\C{\mathcal{C}}
\def\D{\mathcal{D}}

\def\L{\mathcal{L}}
\def\M{\mathcal{M}_{\text{inv}}}

\def\Sig{\Sigma}
\def\glr{\text{GL}(d,\R)}
\def\gl3{\text{GL}(3,\R)}

\makeatother
\def\essinf{{\mathrm{ess inf}}}
\def\esssup{{\mathrm{ess sup}}}
\def\Mk{{\mathcal M}}
\def\Vk{{\mathcal V}}
\def\Ek{{\mathcal E}}

\DeclareMathOperator{\supp}{supp}

  \providecommand{\corollaryname}{Corollary}
  \providecommand{\definitionname}{Definition}
  \providecommand{\lemmaname}{Lemma}
  \providecommand{\propositionname}{Proposition}
  \providecommand{\remarkname}{Remark}
  \providecommand{\theoremname}{Theorem}
\providecommand{\theoremname}{Theorem}

\usepackage{tikz,xcolor,hyperref}

\definecolor{lime}{HTML}{A6CE39}
\DeclareRobustCommand{\orcidicon}{
	\begin{tikzpicture}
	\draw[lime, fill=lime] (0,0) 
	circle [radius=0.16] 
	node[white] {{\fontfamily{qag}\selectfont \tiny ID}};
	\draw[white, fill=white] (-0.0625,0.095) 
	circle [radius=0.007];
	\end{tikzpicture}
	\hspace{-2mm}
}
\foreach \x in {A, ..., Z}{\expandafter\xdef\csname orcid\x\endcsname{\noexpand\href{https://orcid.org/\csname orcidauthor\x\endcsname}
			{\noexpand\orcidicon}}
}

\renewcommand{\le}{\leqslant}
\renewcommand{\leq}{\leqslant}
\renewcommand{\geq}{\geqslant}
\renewcommand{\ge}{\geqslant}
\newcommand{\eps}{\varepsilon}

\author{Balázs Bárány \orcidB{} and Reza Mohammadpour \orcidA{} }

\address{Balázs Bárány, Department of Stochastics, Institute of Mathematics, Budapest University of Technology and Economics, M\oldH{u}egyetem rkp. 3., H-1111 Budapest, Hungary}
\email{barany.balazs@ttk.bme.hu}
\address{Reza Mohammadpour, Department of Mathematics, Uppsala University, Box 480, SE-75106, Uppsala, Sweden.}
\email{rmohammadpour70@gmail.com}
\date{\today}

\subjclass[2020]{28A80, 28D20, 37L30, 37D35}
\keywords{Self-affine set, multifractal analysis,  Lyapunov
exponents, thermodynamic formalism, typicality.}

\usepackage{amsmath}

\begin{document}

\title[Hausdorff Spectrum of Lyapunov Exponents for Self-Affine Systems]{Hausdorff Spectrum of Lyapunov Exponents for Higher-Dimensional Self-Affine Systems}

\maketitle

\begin{abstract}
We study the Hausdorff dimension of projections of level sets of Lyapunov exponents for affine iterated function systems. Assuming that the tuple of linear parts is typical, we obtain a variational formula in terms of pressure, topological entropy, and the Lyapunov dimensions of invariant measures corresponding to Lyapunov exponents. In $\R^3$, we show that this value equals the Hausdorff dimension of the canonical projection of the level set under the strong open set condition. We further extend the result to arbitrary dimensions, showing that the same formula holds for Lebesgue-almost every choice of translation vectors under a suitable contraction assumption.

\color{black}
\end{abstract}

\medskip

\section{Introduction}

The multifractal formalism of Birkhoff averages concerns the size of the set where the Birkhoff averages exist and equal a particular value (for instance, see \cite{FFW, BSS, IJ15, JT21, O, Olsen}). This set is called the \textit{level set}, and its size is determined using either the Hausdorff dimension or the topological entropy introduced by Bowen in \cite{Bowen}. Motivated by the study of the multifractal formalism of Birkhoff averages, the multifractal formalism of certain special sub-additive potentials $\Phi = \left\{\log \phi_n\right\}_{n=1}^{\infty}$ on sub-shifts of finite type has been studied in \cite{feng03, feng09, Moh22-Lyapunov, DGR19}. In this context, $\phi_n(x) = \left\|\prod_{i=0}^{n-1} \A \left(\sigma^i (x)\right)\right\|$, where $\A$ is a continuous function on a finite symbolic space $\Sigma$ taking values in the set of $d \times d$ matrices, $\sigma: \Sigma \to \Sigma$ is a full shift, and $\|\cdot\|$ denotes the operator norm. Moreover, by Oseledets' Multiplicative Ergodic Theorem, Lyapunov exponents exist for every ergodic invariant measure. Since there are many such measures, the Lyapunov exponents may take a huge variety of different values. Therefore, it is natural to study the size of the set of points where all Lyapunov exponents exist and take a particular value. This case is much more challenging than dealing with the size of the top Lyapunov exponent, as the behaviour of the vector of all Lyapunov exponents in matrix cocycles is much more complicated than that of the top Lyapunov exponent. 

Feng and Huang \cite{FH} calculated the entropy spectrum of all Lyapunov exponents for almost additive potentials, improving a result of Barreira and Gelfert in \cite{BG06} on Lyapunov exponents of nonconformal repellers. Recently, \cite{Moh22-entropy, Moh23, Mohammadpour-Varandas, Mohammadpour-survey} calculated the entropy spectrum of all Lyapunov exponents for generic matrix cocycles. 

As mentioned, we determine the size of the level set using either the Hausdorff dimension or the topological entropy. For the Hausdorff dimension, we need to introduce a geometrical structure. The well-known example of such a geometrical structure is a self-affine iterated function system satisfying some separation conditions.

A tuple $\Theta=\left(A_1+v_1, \ldots, A_N+v_N\right)$ of contractive invertible affine self-maps on $\R^d$ is called an {\it affine iterated function system (affine IFS)}. The associated tuple of matrices $\left(A_1, \ldots, A_N\right)$ is therefore an element of $\glr^N$ and satisfies $\max _{i \in\{1, \ldots, N\}}\left\|A_i\right\|<1$. There exists a unique non-empty compact set $X \subset \mathbb{R}^d$ such that
$$
X=\bigcup_{i=1}^N\left(A_i+v_i\right)(X) .
$$

In this case, the set $X$ is called a {\it self-affine set}.  We say that $\Theta$ satisfies a {\it strong open set condition (SOSC)} if there exists an open set $U \subset \mathbb{R}^d$ intersecting $X$ such that the elements of the union $\bigcup_{i=1}^N\left(A_i+v_i\right)(U)$ are pairwise disjoint, contained in $U$ and $X\cap U\neq\emptyset$. Furthermore, $\Theta$ satisfies the {\it strong separation condition (SSC)} if $\left(A_i+v_i\right)(X) \cap\left(A_j+v_j\right)(X)=\emptyset$ whenever $i \neq j$.

Let $\Sigma=\{1, \ldots, N\}^{\mathbb{N}}$ be the collection of all infinite words obtained from integers $\{1, \ldots, N\}$. If $\mathrm{i}=i_1 i_2 \cdots \in \Sigma$, then we define $\left.\mathrm{i}\right|_n=i_1 \cdots i_n$ for all $n \in \mathbb{N}$. The empty word $\left.\mathrm{i}\right|_0$ is denoted by $\varnothing$. Define $\Sigma_n=\left\{\left.\mathrm{i}\right|_n: \mathrm{i} \in \Sigma\right\}$ for all $n \in \mathbb{N}$ and $\Sigma_*=\bigcup_{n \in \mathbb{N}} \Sigma_n \cup\{\varnothing\}$. Thus $\Sigma_*$ is the collection of all finite words. The length of $i \in \Sigma_* \cup \Sigma$ is denoted by $|i|$. The longest common prefix of $i, j \in \Sigma_* \cup \Sigma$ is denoted by $i \wedge j$. The concatenation of two words $i \in \Sigma_*$ and $j \in \Sigma_* \cup \Sigma$ is denoted by $i j$. Let $\sigma$ be the left shift operator defined by $\sigma \mathrm{i}=i_2 i_3 \cdots$ for all $\mathrm{i}=i_1 i_2 \cdots \in \Sigma$. If $\mathrm{i} \in \Sigma_n$ for some $n$, then we set $[\mathrm{i}]=\left\{\mathrm{j} \in \Sigma:\left.\mathrm{j}\right|_n=\mathrm{i}\right\}$. The set $[\mathrm{i}]$ is called a cylinder set. The shift space $\Sigma$ is compact in the topology generated by the cylinder sets. Moreover, the cylinder sets are open and closed in this topology, and they generate the Borel $\sigma$-algebra. Also, we denote by $\mathcal M(\Sigma)$ the set of Borel probability measures on $\Sigma$; we denote by $\M(\Sigma, \sigma)\subset\mathcal M(\Sigma)$ the space of all invariant probability measures; and we denote by $\Ek_\sigma(\Sigma)$ the set of all $\sigma$-invariant ergodic probability measures.  Let $\mu \in \M(\Sigma, \sigma)$ and recall that the measure-theoretic entropy of $\mu$ and $\sigma$ is

$$
h_{\mu}(\sigma)=-\lim _{n \rightarrow \infty} \frac{1}{n} \sum_{i\in \Sigma_n} \mu([i]) \log \mu([\mathrm{i}]).
$$

A probability measure $\mu$ on $(\Sigma, \sigma)$ is \textit{Bernoulli} if there exists a probability vector $\left(p_1, \ldots, p_N\right)$ such that
$$
\mu([i])=p_{i_1} \cdots p_{i_n}
$$
for all $i=i_1 \cdots i_n \in \Sigma_n$ and $n \in \mathbb{N}$. Bernoulli measures are well known to be ergodic. We say that $\mu \in \M(\Sigma, \sigma)$ is an \textit{$n$-step Bernoulli} if it is a Bernoulli measure on $\left(\Sigma, \sigma^n\right)$. In this case, we write
\[
\tilde{\mu}=\frac{1}{n} \sum_{k=0}^{n-1} \mu \circ \sigma^{-k} \]
and note that $\tilde{\mu} \in \Ek_\sigma(\Sigma)$ and $ h_{\mu}\left( \sigma^n\right)=n h_{\tilde{\mu}}(\sigma)$.

Write $\A_{\mathrm{i}}=A_{i_1} \cdots A_{i_n}$ for all $\mathrm{i}=i_1 \cdots i_n \in \Sigma_n$ and $n \in \mathbb{N}$. The canonical projection $\pi: \Sigma \rightarrow X$ is defined by $\pi(\mathrm{i})=\sum_{n=1}^{\infty} \A_{\left.\mathrm{i}\right|_{n-1}} v_{i_n}$ for all $\mathrm{i}=$ $i_1 i_2 \cdots \in \Sigma$. When we need to emphasise how the canonical projection depends on the translation parameters, we denote it by $\pi_v$, where $v=(v_1,\ldots,v_N)$. It is easy to see that $\pi(\Sigma)=X$. If $\mu$ is a measure on $\Sigma$, then we denote the push-forward measure of $\mu$ under $\pi$ by $\pi_* \mu=\mu \circ \pi^{-1}$.

In this paper, we are interested in the Hausdorff dimension of \(\pi E(\vec{\alpha})\), where \(E(\vec{\alpha})\) is an \(\vec{\alpha}\)-level set defined as follows:
\[E(\vec{\alpha}) = \bigg\{ I \in \Sigma : \lim_{n \to \infty} \frac{1}{n} \log \sigma_i (\mathcal{A}_{I|_{n}}) = \alpha_i \text{ for } i = 1, 2, \ldots, d \bigg\},\]
where \(\sigma_1, \ldots, \sigma_d\) are singular values listed in decreasing order according to multiplicity. We also define the \textit{Lyapunov spectrum}
\[\vec{L}_{\A} = \bigg\{ \vec{\alpha} \in \mathbb{R}^d : \exists I \in \Sigma \text{ such that } \lim_{n \to \infty} \frac{1}{n} \log \sigma_i (\mathcal{A}_{I|_{n}}) = \alpha_i \text{ for } i = 1, 2, \ldots, d \bigg\}.\]
For simplicity, we denote $\vec{L}:=\vec{L}_{\A}$. By Park \cite[Theorem~D]{park2020quasi}, $\vec{L}_{\A}$ is closed and convex. Let $\mathring{\vec{L}}$ denote the relative interior of $\vec{L}$.

Several results exist for the conformal case, which is a particular case of Birkhoff averages (for instance, see \cite{BSS, FFW, Olsen}). In this paper, we focus on the non-conformal case. In the two-dimensional setting, B\'ar\'any, Jordan, K\"aenm\"aki and Rams \cite{BJKR} calculated the Hausdorff dimension \(\pi E(\vec{\alpha})\) for affine IFSs satisfying the strong open set condition and the coordinate-wise exponential separation under the assumption that the set of matrices is strongly irreducible and that the generated subgroup of the normalised matrices is not relatively compact. Imbierski, Kalle and Mohammadpour \cite{IKM} calculated the Hausdorff dimension \(\pi E(\vec{\alpha})\) for diagonally affine IFSs satisfying the strong open set condition under the assumption that the set of matrices is dominated. As far as we are aware, there is no result for the Hausdorff dimension \(\pi E(\vec{\alpha})\) in the \(d\)-dimensional non-conformal situation, when $d>2$.

In this paper, we consider tuples $(A_1,\ldots, A_k)\in \glr^k$ that are \textit{typical} (see Section~\ref{typical cocycles} for the definition). The typicality assumptions can be viewed as suitable generalisations of the notions of proximality and strong irreducibility in the context of random products of matrices (cf. \cite{Furstenberg1963, Viana-book}). Bonatti and Viana \cite{bonatti2004lyapunov} introduced the notion of typicality and showed that the set of typical tuples is open, dense, and has full Lebesgue measure. We then calculate the Hausdorff dimension of $\pi E(\vec{\alpha})$ for $\R^d$-affine iterated function systems satisfying the typicality assumption. One of our main tools is Theorem \ref{thm:everyergodic}, which can be seen as an extension of the result of Jordan, Pollicott and Simon \cite{JordanPollicottSimon2007} in the sense that the dimension result holds simultaneously for every ergodic measure over the same full-measure set of parameters.

The paper is outlined as follows. We state our main results in Section \ref{sec: results}. In Section~\ref{typical cocycles}, we introduce typical cocycles and recall the pressure and variational-principle results used throughout the paper. In Section~\ref{upper-bound}, we prove the upper bound for the Hausdorff spectrum. In Section~\ref{sec:dominated-case}, we treat the dominated case and obtain the corresponding lower bound. In Section~\ref{sec: simultaneous-transversality}, we show that, for almost every translation vector, the Hausdorff dimension of the projection of every ergodic measure equals the Lyapunov dimension under the simultaneous transversality assumption. Finally, Sections~\ref{proof-main-result1} and~\ref{proof-main-result2} are devoted to the proofs of Theorems~\ref{t:main-result1} and~\ref{t:main-result2}, respectively.

\section{Statement of the main result}\label{sec: results}


Let $A \in \glr.$ We define \textit{Falconer's singular value function} $\varphi^{s}(A)$ as follows.  Let $k \in\{0, \ldots, d-1\}$ and $k \leq s<k+1$. Then,
$$
\varphi^{s}(A)=\sigma_{1}(A) \cdots \sigma_{k}(A) \sigma_{k+1}(A)^{s-k},
$$

and if $s\geq d,$ then $\varphi^{s}(A)=(\det(A))^{\frac{s}{d}}.$ For $s:=(s_{1}, \cdots,  s_{d})\in \R^d$, we define the \textit{generalised singular value function} $\psi^{s_{1}, \ldots, s_{d}}(A): \mathbb{R}^{d \times d} \rightarrow[0, \infty)$ as
$$
\psi^{s_{1}, \ldots, s_{d}}(A):=\sigma_{1}(A)^{s_{1}} \cdots \sigma_{d}(A)^{s_{d}}=\left(\prod_{m=1}^{d-1}\left\|A^{\wedge m}\right\|^{s_{m}-s_{m+1}}\right)\left\|A^{\wedge d}\right\|^{s_{d}}.
$$

Let us define a map $s'\colon[0,\infty)\to\R^d$ by
\begin{equation}\label{definition_s'}
s'(s)=\begin{cases}
    (\underbrace{1, \ldots, 1}_{\lfloor s\rfloor\text{-times }}, s-\lfloor s\rfloor, \underbrace{0, \ldots, 0}_{\substack{d-(\lfloor s\rfloor+1)\\\text{-times }}}) & \text{ for }0\leq s< d\\
    (s/d,\ldots,s/d) & \text{ for }s\geq d.
\end{cases}
\end{equation}
Then the singular value function $\varphi^{s}(A)$ coincides with the generalized singular value function $\psi^{s'(s)}(A)$. 
Let us introduce the notation
\[
\Psi(A)=\left(\log \sigma_{1}(A), \ldots, \log \sigma_{d}(A)\right).
\]
Using this, we have the alternative representation
$$
\log\psi^q(A)=\langle q,\Psi(A)\rangle
$$
for every $q\in\R^d$.


        

Let $\mu \in \M(\Sigma, \sigma).$ We denote 
\[ \vec{\chi}(\mu, \A):=(\chi_{1}(\mu, \A), \ldots, \chi_{d}(\mu, \A)),\]
where $\chi_{i}(\mu, \A):=\lim_{n \to \infty} \frac{1}{n} \int \log \sigma_{i}(\A_{J|_{n}}) d\mu(J).$ For simplicity, we abbreviate $\vec{\chi}(\mu,\A)$ to $\vec{\chi}(\mu)$ and $\chi_{i}(\mu, \A)$ to $\chi_{i}(\mu)$ whenever there is no ambiguity.

 The \textit{Lyapunov dimension} of $\mu \in \M(\Sigma, \sigma)$  is defined to be $$ \dim_{\text{L}} \mu=\min _{k \in\{1, \ldots, d\}}\left\{k-1+\frac{h_\mu(\sigma)-\sum_{i=1}^{k-1} \chi_i(\mu)}{\chi_k(\mu)}\right\}.$$


The topological pressure of the generalised singular value potential  is defined by
\begin{equation}\label{eq:P*}
P^*\left(\log \psi^q (\A)\right):=\limsup _{n \rightarrow \infty} \frac{1}{n} \log Z_n(q), \quad \forall q \in \mathbb{R}^d,
\end{equation}
where
$$
Z_n(q):=\sum_{I \in \Sigma_n} \psi^q(\A_{I}).
$$

The function $Z_n(q)$ is often referred to as the partition function. In the case that $(A_1, \ldots, A_N)$ is typical (see Section~\ref{typical cocycles}), one can replace limsup by a limit, and we denote it by $P(\log \psi^q(\A)$) (cf. Section~\ref{typical cocycles}).

We denote by $h_{\text{top}}(\sigma, E(\vec{\alpha}))$ the topological entropy (in the sense of \cite{Bowen}) of the level set. We are now ready to formulate our main theorems.
The results determine the Hausdorff dimensions of the canonical projections
of $E(\vec{\alpha})$.

\begin{maintheorem}\label{t:main-result1}Let $\left(A_1 + v_1, \ldots, A_N + v_N\right)$ be an affine IFS on $\mathbb{R}^3$ satisfying the SOSC, where $\{A_1, \ldots, A_N\} \subset \mathrm{GL}(3,\mathbb{R})$. Assume that $\left(A_1, \ldots, A_N\right)$ is typical. Then,
\[
\begin{aligned}
\dim_{\mathrm{H}}\big(\pi E(\vec{\alpha})\big) 
&= \sup\left\{ s \geq 0 : \inf_{q \in \mathbb{R}^3} \Big( P\big(\log \psi^{s'(s)+q}(\A)\big) - \langle q, \vec{\alpha} \rangle \Big) \geq 0 \right\} \\
&= \min_{k \in \{1, 2, 3\}} \left\{ k - 1 + \frac{h_{\mathrm{top}}(E(\vec{\alpha})) - \sum_{i=1}^{k-1} \alpha_i}{\alpha_k} \right\} \\
&= \sup \left\{ \dim_{\mathrm{L}} \mu : \mu \in \M(\Sigma, \sigma) \text{ and } \vec{\chi}(\mu) = \vec{\alpha} \right\},
\end{aligned}
\]
for all $\vec{\alpha} \in \mathring{\vec{L}}$.
\end{maintheorem}
Theorem \ref{t:main-result1} gives the Hausdorff dimension of the projected level sets under the SOSC in dimension three, assuming that the set of matrices is typical, which is an open and dense (generic) condition; see \cite{bonatti2004lyapunov}. The following theorem extends the result to arbitrary dimensions and replaces the SOSC assumption with a suitable contraction assumption on the linear parts by showing that the same dimension formula holds for Lebesgue-almost every choice of the translation vectors.

\begin{maintheorem}\label{t:main-result2} Let $\{A_1, \ldots, A_N\} \subset \glr$ be such that $\max_{i\neq j}\|A_i\|+\|A_j\|< 1$, and  $\left(A_1, \ldots, A_N\right)$ is typical. Then, for Lebesgue-almost all $v=(v_i)_{i=1}^N$, we have for the affine IFS $\Theta_v=\left(A_1 + v_1, \ldots, A_N + v_N\right)$ in $\R^d$ that
\[
\begin{aligned}
\dim_{\mathrm{H}}\big(\pi_v E(\vec{\alpha})\big) 
&= \sup\left\{ s \geq 0 : \inf_{q \in \mathbb{R}^d} \Big( P\big(\log \psi^{s'(s)+q}(\mathcal{A})\big) - \langle q, \vec{\alpha} \rangle \Big) \geq 0 \right\} \\
&= \min_{k \in \{1, \ldots, d\}} \left\{ k - 1 + \frac{h_{\mathrm{top}}(E(\vec{\alpha})) - \sum_{i=1}^{k-1} \alpha_i}{\alpha_k} \right\} \\
&= \sup \left\{ \dim_{\mathrm{L}} \mu : \mu \in \M(\Sigma, \sigma) \text{ and } \vec{\chi}(\mu) = \vec{\alpha} \right\},
\end{aligned}
\]
for all $\vec{\alpha} \in \mathring{\vec{L}}$, where $\pi_v$ denotes the canonical projection of $\Theta_v$.
\end{maintheorem}

\section{Typical matrices}\label{typical cocycles}

The eccentricity of a linear map $B \in \glr$ is defined by 
\[\operatorname{Ecc}(B)=\min _{1 \leq i<d} \frac{\sigma_{i}(B)}{\sigma_{i+1}(B)}.\]
We denote by $\Grass(k, d)$ the $k$-th Grassmannian, i.e., the set of all $k$-dimensional subspaces of $\R^d$.

\begin{defn}\label{typical1}
We say that $(A_1,\ldots,A_k)\in\glr^k$ is \textit{typical} if the semigroup generated by $\{A_1,\ldots,A_k\}$ satisfies the following properties:
\begin{itemize}
\item[(i)]$\{\A_I \}_{I \in \Sigma_{\ast}}$ contains elements with arbitrarily large
eccentricity;
\item[(ii)] for any $1\leq i \leq d-1,$ and any $F\in \Grass(i,d)$, and any finite family $G_{1},\ldots,G_{n} \in \Grass(d-i,d),$ there exists $K \in \Sigma_{*}$ such that $\A_{K}\left(F\right) \cap G_{j}=\{0\}$ for
every $1\leq j \leq n.$
\end{itemize}

\end{defn}
The above definition is borrowed from \cite{Moh23}, where it is based on \cite[Section 8]{Viana-book} and \cite[Subsection A.4.5]{AV07}. Avila and Viana \cite[Subsection A.4.5]{AV07} showed that this definition is equivalent to the general definitions of typical cocycles given in \cite[Definition 1.2]{AV07}. Avila and Viana \cite{AV07, AV07-acta}, as well as Bonatti and Viana \cite{bonatti2004lyapunov}, showed that the set of typical cocycles is open and dense. Their works are considered an extension of Furstenberg's work \cite{Furstenberg1963}, in which he proved the simplicity of the top Lyapunov exponent under the assumptions of proximality and strong irreducibility (see \cite{Viana-book, Viana2020} for more details).


\begin{thm}\label{vp-for the generalized sing}
Assume that $(A_1, \ldots, A_k)\in \glr^k$ is typical. Then in the definition \eqref{eq:P*} of $P^*(\log \psi^q(\mathcal{A}))$, the limsup can be replaced by a limit for any $q\in \R^{d}$, and the following variational principle holds:
\[ P(\log \psi^{q}(\A))=\sup\bigg\{ h_{\mu}(\sigma)+\lim_{n\to \infty} \frac{1}{n} \int \log \psi^{q}(\A_{x|_{n}}) d\mu(x): \mu \in \M(\Si, \sigma) \bigg\} \]
for any $q \in \R^d.$
\end{thm}
\begin{proof}
 It follows from \cite{Moh23}.
\end{proof}
We say that  $\A=(A_{1},\ldots,A_{k})\in \glr^{k}$ is \textit{dominated}  if and only if  there exist $C>0$  and $0<\tau<1$ such that for any $I \in \Sigma_*$, we have
$$
\frac{\sigma_{i+1}\left(\A_{I}\right)}{\sigma_i\left(\A_{I}\right)}<C \tau^{|I|}, \qquad \text{for all } i=1, \ldots, d-1.
$$
Note that the domination property can also be characterised in terms of the existence of invariant cone fields. That is, it follows from Bochi and Gourmelon \cite[Theorem B]{BG} that for every $k=1,\ldots,d-1$, there exists a closed proper subset $\mathcal{C}_k\subset\mathbb{P}(\R^{\wedge k})$  such that a $(\binom{d}{k}-1)$-dimensional plane in $\R^{\wedge k}$ is transverse to all elements of $\mathcal{C}_k$, and for every $i=1,\ldots,N$
$$
A_{i}^{\wedge k}\mathcal{C}_k\subseteq\mathring{\mathcal{C}_k},
$$
where $\mathring{\mathcal{C}_k}$ denotes the interior of $\mathcal{C}_k$. By Bochi and Morris \cite[Lemma~2.2]{Bochi-Morris}, there exists a constant $c>0$ such that for every $k=1,\ldots,d-1$ and $W\in\bigcup_{i=1}^NA_{i}^{\wedge k}\mathcal{C}_k$
\begin{equation}\label{eq:BM}
\|A_i^{\wedge k}|W\|\geq c\|A_i^{\wedge k}\|.
\end{equation}

For every $k=1,\ldots,d-1$, there exists a H\"older continuous map $W_k\colon\Sigma\to \mathbb{P}(\R^{\wedge k})$ such that for every $I=i_1i_2\cdots\in\Sigma$
$$
W_k(I)=A^{\wedge k}_{i_1}W_k(\sigma I),
$$
see \cite[Theorem~4.11]{CP}. Let us define the potential $\Phi_k(I):=\log\|A_{i_1}^{\wedge k}|W_k(\sigma I)\|$. It follows from \eqref{eq:BM} and the fact that $W_k(I)\in\bigcup_{i=1}^NA_{i}^{\wedge k}\mathcal{C}_k$ that there exists $C>0$ such that for every $I\in\Sigma$ and every $n\in\N$
and 
\begin{equation}\label{eq:approximatewithpotential0}
\left|\log\|\A_{I|_n}^{\wedge k}\|-S_n\Phi_k(I)\right|\leq C,
\end{equation}
where $S_n\Phi_k(I)=\sum_{\ell=0}^{n-1}\Phi_k(\sigma^\ell I)$ is the $n$-th Birkhoff sum of the potential $\Phi_k$. Denote
$$
\Phi(\A)=\left(\Phi_1(\A),\frac{\Phi_2(\A)}{\Phi_1(\A)},\ldots,\frac{\Phi_d(\A)}{\Phi_{d-1}(\A)}\right).
$$
Using \eqref{eq:approximatewithpotential0}, we get that there exists $C'>0$ such that for every $n\in\N$ and every $I\in\Sigma$
\begin{equation}\label{eq:approximatewithpotential}
\left|\Psi(\A_{I|_n})-S_n\Phi(I)\right|\leq C.
\end{equation}

The following theorem shows that typical systems contain sufficiently large dominated subsystems in some proper sense.

 \begin{thm}[{\cite[Corollary 4.6]{Moh22-entropy}}] \label{dominated-one-step-cocycle}
Assume that $(A_1, \ldots, A_k)\in \glr^k$ is typical. Then, there exists $K_0 \in \N$ such that for every $n\in \N$ and $I \in \Sigma_n$ there exist $J_2=J_2(I)$ and $J_1=J_1(I)$ with $|J_{i}|\leq K_0$ for $i=1,2$ such that the tuple \[\left(\mathcal{A}_{\mathrm{k}}\right)_{\mathrm{k} \in \Sigma_n^{\mathcal{D}}}, \quad \text{where } \Sigma_n^{\mathcal{D}}:=\left\{J_1(I) IJ_2(I): I \in \Sigma_n\right\},\]
is dominated.
\end{thm} 
Note that for each $I \in \Sigma_n^{\mathcal{D}}$, $|I| \in [n, n+2K_0].$

We define a pressure on the dominated subsystem $\Sigma_n^{\mathcal{D}}$ by setting \color{black}
\[P_{n, \mathcal{D}}(\log\varphi):=\lim _{m \rightarrow \infty} \frac{1}{m} \log \sum_{I_{1}, \ldots, I_{m} \in \Sigma_n^{\mathcal{D}}}\varphi(I_1 \ldots I_m),\]
\color{black}whenever $\varphi: \Sigma_{\ast} \rightarrow \mathbb{R}_{\geq 0}$ is sub-multiplicative, i.e.,
$$
\varphi(\mathrm{I}) \varphi(\mathrm{J}) \geq \varphi(\mathrm{IJ}).
$$
for all $I, J \in \Sigma_{\ast}$ with $IJ \in \Sigma_{\ast}.$

If $(A_1, \ldots, A_k)\in \glr^{k}$ is dominated, then $\{\log \sigma_{i}(\A_{I})\}$ is almost additive for all $I \in \Sigma_*$ and $ i=1, \ldots, d$ by \cite[Proposition 5.8]{Moh22-Lyapunov} (see also \cite[Remark 4.2]{Moh22-entropy}). That is, there exists a constant $c>0$ such that for any $I,J\in\Sigma_*$ and $i=1,\ldots,d$
$$
\log\sigma_{i}(\A_{I})+\log\sigma_{i}(\A_{J})-c\leq\log\sigma_{i}(\A_{IJ})\leq\log\sigma_{i}(\A_{I})+\log\sigma_{i}(\A_{J})+c.
$$

 \begin{lem}\label{psi is almost additive}
 Assume that $(A_1, \ldots, A_k) \in \glr^k$ is dominated. Then, the potential $\{\langle q, \Psi(\A_{I}) \rangle\}$ is almost additive  for any  $I\in \Sigma_*$ and $q \in \R^d.$ 
\end{lem}
\begin{proof}
It follows from \cite[Proposition 5.8]{Moh22-Lyapunov}.
\end{proof}

Assume that $(A_1, \ldots, A_k)\in \glr^k$ is typical. By Theorem \ref{dominated-one-step-cocycle}, $\mathcal{B}:=\left(\mathcal{A}_{\mathrm{k}}\right)_{\mathrm{k} \in \Sigma_n^{\mathcal{D}}}$ is dominated.  We also consider the full shift $((\Sigma_n^{\mathcal{D}})^{\N},f)$. It is easy to see that $(\Sigma_n^{\mathcal{D}})^{\N} \subset \Si$.

\begin{thm}[{\cite[Theorem 4.8]{Moh22-entropy}}]\label{continuity_potential}Assume that  $(A_1, \ldots, A_k) \in \glr^k$ is typical. Then, 
\[\lim _{n \rightarrow \infty} \frac{1}{n} P_{n, \mathcal{D}}(\langle q, \Psi(\mathcal{B}) \rangle )=P(\log \psi^{q}(\mathcal{A})),\]
for each $q\in\R^d$ (uniformly on compact subsets).
\end{thm}
\begin{prop}\label{relation between entropies and LE}
 Suppose that $(A_1, \ldots, A_k)\in \glr^k$ is typical. Then,  there exists $K_0 \ge 1$ so that, for each $n>2K_0$ there is a full shift $((\Sigma_{n}^{\D})^{\N}, f)$ and a dominated $\mathcal{B}:=\left(\mathcal{A}_{\mathrm{k}}\right)_{\mathrm{k} \in \Sigma_n^{\mathcal{D}}}$ (depending on $n$)  satisfying the following:

For any $\mu' \in \M((\Sigma_{n}^{\D})^{\N}, f)$, there is $\mu \in \M(\Si, T)$ such that
\begin{equation}\label{entropies-relation}
 h_{\mu'}(f) \leq (n+2K_0)h_{\mu}(T)+\frac {n+2K_0}n \log (2K_0+1),
\end{equation}
and 
\color{black}
\begin{equation}\label{LE-relations}
\lim_{m \to \infty} \frac{1}{m} \int \langle q, \Psi(\B_{x|_m}) \rangle d\mu'(x) \leq (n+2K_{0})\lim_{m\to \infty} \frac{1}{m} \int \log \psi^{q}(\A_{x|_m}) d\mu(x).
\end{equation}
\color{black}
\end{prop}

\begin{proof}
    It follows from \cite{Moh23}.
\end{proof}
We denote by $\vec{L}_{\mathcal{B}}$ the Lyapunov spectrum corresponding to the matrices $\mathcal{B}$, and set
\[
\Omega := \{ \vec{\chi}(\mu, \mathcal{B}) : \mu \in \M((\Sigma_{n}^{\D})^{\N}, f) \}.\]
Lemma \ref{psi is almost additive}
  ensures that 
  the family of potentials $\{\log \sigma_i(\B_I)\}$  is almost additive. By  \cite[Theorem 5.2 item (1)]{FH}, $\mathring{\vec{L}}_{\B} \subset \Omega$.
  Thus, since such a cocycle is dominated, combining \cite[Theorem 5.4]{Moh22-entropy} 
  and 
  \cite[Theorem 5.2]{FH},
  one concludes that 
  \medskip
\begin{equation}\label{eq:topological-entropy-level-set}
\begin{aligned}
h_{\mathrm{top}}(f, E^{n, \mathcal{D}}(\vec{\alpha}))
    & = \inf _{q \in \mathbb{R}^{3}}
    \left\{P\left(\log \psi^{q}(\mathcal B) \right)
    - \langle q, \vec{\alpha} \rangle \right\} \\
    & = \sup\big\{ h_{\mu}(f) \colon
    \mu \in \M((\Sigma_{n}^{\D})^{\N}, f), \;
    \chi_i(\mu, \mathcal B)=\alpha_i,
    \forall\, 1\le i\le d\big\}.
\end{aligned}
\end{equation}
for every $\alpha \in \mathring{\vec{L}}_{\B}$, where
\[
E^{n, \mathcal{D}}(\vec{\alpha}):=\bigg\{ I\in (\Sigma_{n}^{\D})^{\N}: \lim_{m\to \infty}\frac{1}{m}\log \sigma_{i}(\mathcal{B}_{I_{|m}})=\alpha_i \text{ for } 1\le i \le d \bigg\}.
\]

\begin{lem}\label{lem: the relation between topological entropy}
  Suppose that  $(A_1, \ldots, A_k)\in \glr^k$ is typical. Then,  there exists $K_0 \ge 1$ so that, for each $n>2K_0$, there is a full shift $((\Sigma_{n}^{\D})^{\N}, f)$ and a dominated $\mathcal{B}:=\left(\mathcal{A}_{\mathrm{k}}\right)_{\mathrm{k} \in \Sigma_n^{\mathcal{D}}}$ (depending on $n$) such that
\begin{equation}\label{eq:lemrel}\limsup_{n\to\infty}\frac{1}{n}h_{\mathrm{top}}\left(f, E^{n, \mathcal{D}}(n \vec{\alpha})\right) \leq h_{\mathrm {top}}(\sigma, E(\alpha)). \end{equation}
\end{lem}

\begin{proof}
    By \cite{Moh23}, for any $\vec\alpha\in\mathring{\vec L}$
\begin{equation}\label{restricted}
     h_{\text{top}} \Big(\sigma, E(\vec \alpha)\Big)=\sup\big\{ h_{\mu}(\sigma) \colon   \mu \in \M(\Sigma, \sigma), \; 
    \chi_i(\mu, \A)=\alpha_i, \forall 1\le i \le d\big\}.
\end{equation}

\smallskip
Fix $\alpha \in \mathring{\vec{L}}$.
By using 
Proposition \ref{relation between entropies and LE} and \eqref{eq:topological-entropy-level-set} , we have 
\[ 
\begin{aligned}
   h_{\text{top}} (f, E^{n, \mathcal{D}}( n \vec{\alpha})) 
   & = \sup\big\{ h_{\mu}(f) \colon   \mu \in \M( (\Sigma_{n}^{\D})^{\N}, f), \; 
 \chi_i  (\mu, \mathcal B)=  n \alpha_i,  \forall 1\le i \le 3\big\} \\
    & \le  (n+2K_0)\sup\Big\{ \, h_{\nu}(\sigma) + \frac 1n \log (2K_0+1) \colon \\
    & \qquad \qquad    \nu \in \M(\Sigma, \sigma), \; 
 \chi_i(\nu, {\mathcal A}) \color{black} \in \Big[\frac{n}{n+2K_0} \alpha_i, \alpha_i \Big],
     \forall 1\le i \le d\Big\}\\
     & \le n\sup\Big\{ \, h_{\nu}(\sigma) \colon \\
    & \qquad \qquad    \nu \in \M(\Sigma, \sigma), \; 
 \chi_i(\nu, {\mathcal A}) \color{black} \in \Big[\frac{n}{n+2K_0} \alpha_i, \alpha_i \Big],
     \forall 1\le i \le d\Big\}+C
\end{aligned}
\]
 for some $C>0$ independent of $n$. Therefore, by \eqref{restricted} and the upper semi-continuity of the measure-theoretic entropy, \eqref{eq:lemrel} follows.
\end{proof}
\section{Upper bound}\label{upper-bound}
In this section, we prove the upper bound of the main result. 

\begin{thm}\label{HD-one side}Let $\left(A_1+v_1, \ldots, A_N+v_N\right)$ be an affine IFS on $\mathbb{R}^d$. If $(A_1, \ldots, A_N)$ is typical, then
\[\dim_{\mathrm{H}}\left(\pi E(\vec{\alpha})\right) \leq \sup \left\{s \geq 0: \inf _{q \in \mathbb{R}^{d}}\left\{P^*\left(\log \psi^{s'(s)+q}(\A)\right)- \langle q,  \vec{\alpha} \rangle  \right\} \geq 0\right\}
\]
for all $\alpha \in \vec{L}$, where $s'(s)$ is defined \eqref{definition_s'}.

\end{thm}
\begin{proof}

 For $\vec{\alpha} \in \vec{L},$ we define
\[G(m, r):=\left\{I \in \Sigma_{m} :\text{ and }  \left|\frac{1}{m}\log \sigma_{i}(\mathcal{A}_{I})-\alpha_i\right|<\frac{1}{r}  \text{ for }i=1, \ldots, d \right\}.\]
Note that
\[E(\vec{\alpha}) \subset \bigcap_{r=1}^{\infty} \bigcup_{n=1}^{\infty} \bigcap_{m=n}^{\infty} \bigcup_{I \in G(m, r)}[I].\]
  Therefore, for every
$i \in  G(m, r)$, we have  
\begin{equation}\label{equ11}
-\sum_{i=1}^{d} \frac{|q_i|}{r}\leq \langle q,  \Psi(\A_{i})-m\vec{\alpha} \rangle.
\end{equation}

 
 Let $s_{0}(\vec{\alpha})=:\sup\left\{s\geq 0: \inf_{q\in \R^{d}}\left\{P^*(\log \psi^{s'(s)+q}(\A))- \langle \vec{\alpha}, q \rangle \right\}\geq 0\right\},$ and choose $s>s_{0}(\vec{\alpha})$. Therefore, there is $q:=q(\vec{\alpha}, s)$ such that $P^*(\log \psi^{s'(s)+q}(\A))< \langle q, \vec{\alpha} \rangle .$ Assume that $\epsilon>0$ is so small such that there exists $\gamma>0$ such that
\begin{equation}\label{equ22}
\sum_{I\in \Sig_n}\psi^{s'(s)+q}(\mathcal{A}_{I})<e^{n(\langle q, \vec{\alpha} \rangle+\epsilon)}
\end{equation}
for all $n\geq -\log \gamma$.

 Setting $\Gamma:=\max_{x\in \Sigma} \|\mathcal{A}(x)\|.$ Therefore,  \begin{equation}\label{eq33}
\varphi^{s+c}(\mathcal{A}_{I})\leq \varphi^{s}(\mathcal{A}_{I})e^{c |I|\Gamma},
\end{equation}
for all $c\geq 0$.

For $\delta< \gamma,$ by \eqref{equ11},\eqref{equ22} and \eqref{eq33}, we have
$$
\begin{aligned}
\mathcal{H}_{\delta}^{s-\frac{\sum_{i=1}^{d} |q_i|}{r\Gamma}}(\pi E(\vec{\alpha}))&\leq \sum_{m=\lceil -\log \delta \rceil}^{\infty}\sum_{I\in G(m, r)} \varphi^{s-\frac{\sum_{i=1}^{d} |q_i|}{r\Gamma}}(\mathcal{A}_{I})\\
& \leq \sum_{m=\lceil -\log \delta \rceil}^{\infty}\sum_{I\in G(m, r)} \varphi^{s}(\mathcal{A}_{I})e^{\frac{-2m\sum_{i=1}^{d} |q_i|}{r}}\\
& \leq \sum_{m=\lceil -\log \delta \rceil}^{\infty}\sum_{I\in G(m, r)} \varphi^{s}(\mathcal{A}_{I})e^{\langle q,  \Psi(\A_I)- m\vec{\alpha} \rangle}\\
&=\sum_{m=\lceil -\log \delta \rceil}^{\infty}e^{m \langle q, \vec{\alpha} \rangle }\sum_{I\in G(m, r)} \psi^{s'(s)-q}(\mathcal{A}_{I})\\
&\leq \sum_{m=\lceil -\log \delta \rceil}^{\infty} e^{-m\epsilon} \to 0
\end{aligned}
$$
when $\delta \to 0$ for all $r\geq 1.$ Therefore, $\dim_{ \mathrm{H}}(\pi E(\vec{\alpha}))\leq s-\frac{\sum_{i=1}^{d} |q_i|}{r\Gamma}.$ We get $\dim_{\mathrm{H}}(\pi E(\vec{\alpha}))\leq s_{0}(\alpha)$ as $r\geq 1$ and $s>s_{0}(\alpha).$
\end{proof}



\section{Dominated case}\label{sec:dominated-case}

Let us note that if $(A_1, \ldots, A_N)\in \glr^{N}$ is dominated then $P^*\left(\log \psi^{q}(\A)\right)=P\left(\log \psi^{q}(\A)\right)$ for any $q\in\R^d$ by Lemma \ref{psi is almost additive}.

\begin{prop}\label{existence of Bernoulli measure}
    Let $\Phi:=\left(A_1+v_i, \ldots, A_N+v_N\right)$ be an affine IFS on $\mathbb{R}^d$, where\linebreak $\{A_1, \ldots, A_{N}\} \subset \glr$. Suppose that $(A_1, \ldots, A_N)\in \glr^{N}$ is dominated. If $\alpha \in \mathring{\vec{L}}$, then for every

$$
s<\sup\left \{s \geq 0: \inf _{q \in \mathbb{R}^{d}}\left\{P\left(\log \psi^{s'(s)+q}(\A)\right)- \langle q ,\vec{\alpha} \rangle \right\} \geq 0\right\},
$$
then, there exists a fully supported $n$-step Bernoulli measure $\nu$ such that $$\operatorname{dim}_{\mathrm{L}}(\nu)= \operatorname{dim}_{\mathrm{L}}(\tilde{\nu}) \geq s$$ and $\vec{\chi}(\tilde{\nu})=\vec{\alpha}$.
\end{prop} 
\begin{proof}
 The proof is inspired by \cite[Proposition 4.3]{BJKR}. Note that by \eqref{eq:approximatewithpotential}
 $$ P\left(\log \psi^{s'(s)+q}(\A)\right)- \langle q ,\vec{\alpha} \rangle=P\left(\log \varphi^{s}(\A)+  \langle  q , \Psi(\A) -\vec{\alpha} \rangle\right)=P(\langle s'(s)+q,\Phi\rangle-\langle q,\vec{\alpha}\rangle).$$

 Assume that 
$s< \sup\left\{s \geq 0: \inf _{q \in \mathbb{R}^{d}}P(\langle s'(s)+q,\Phi\rangle-\langle q,\vec{\alpha}\rangle)\leq0 \right\}$. Since for any $q \in \mathbb{R}^d$,  $s \rightarrow P\left(\log \varphi^{s}(\A)+  \langle  q , \Psi(\A) -\vec{\alpha} \rangle\right) $ is strictly decreasing we have that

$$
\inf \left\{P(\langle s'(s)+q,\Phi\rangle-\langle q,\vec{\alpha}\rangle) : q \in \mathbb{R}^d\right\}>0.
$$

Since $\vec{\alpha} \in \mathring{\vec{L}}$ we can find $\eta>0$ such that for any $q \in \mathbb{R}^d$ with $|q|=1$ we can find an invariant measure $\mu$ such that $\vec{\chi}(\mu)-\vec{\alpha}=\eta q$. Therefore, by the variational principle for any $q \in \mathbb{R}^d$

$$
P(\langle s'(s)+q,\Phi\rangle-\langle q,\vec{\alpha}\rangle)  \geq h_{\mu}(\sigma)+\vec{\chi}(\mu)+\eta|q|
$$
where $h_{\mu}(\sigma)+\vec{\chi}(\mu)$ is bounded uniformly below for all invariant measures. Thus, for
$$
\delta=\inf \left\{P(\langle s'(s)+q,\Phi\rangle-\langle q,\vec{\alpha}\rangle) : q \in \mathbb{R}^d\right\}>0
$$
we can choose $q_0>0$ so that
$$
P(\langle s'(s)+q,\Phi\rangle-\langle q,\vec{\alpha}\rangle)  \geq 3 \delta
$$
whenever $|q| \geq q_0$. We fix $\varepsilon>0$ such that
$$
\varepsilon (q_0+s)<\frac{\delta}{4}.
$$
Since $\Phi$ is (H\"older) continuous, we can find functions $\Phi_n$, which is constant on $n$th level cylinders such that
 $$
 \left\|S_n \Phi-\Phi_n\right\|_{\infty} \leq n \varepsilon.
 $$
We now work with the pressure for $\sigma^n$, which we denote by $P_n$.
$$
P(\langle s'(s)+q,\Phi\rangle-\langle q,\vec{\alpha}\rangle)=\frac{1}{n}P_n(\langle s'(s)+q,S_n\Phi\rangle-\langle q,n\vec{\alpha}\rangle).
$$
Thus, we have
$$
\inf \left\{P_n(\langle s'(s)+q,S_n\Phi\rangle-\langle q,n\vec{\alpha}\rangle):|q| \leq q_0\right\}=n \delta
$$
and
$$
\inf \left\{P_n(\langle s'(s)+q,S_n\Phi\rangle-\langle q,n\vec{\alpha}\rangle):|q|=q_0\right\} \geq 3 n \delta .
$$

Since $\varepsilon(q_0+s)<\delta / 4$, we see that
\begin{equation}\label{bound for the pressure}
    \min \left\{P_n(\langle s'(s)+q,\Phi_n\rangle-\langle q,n\vec{\alpha}\rangle):|q| \leq q_0\right\} \in\left[\frac{3 n \delta}{4}, \frac{5 n \delta}{4}\right]
\end{equation}
and
$$
\min \left\{P_n(\langle s'(s)+q,\Phi_n\rangle-\langle q,n\vec{\alpha}\rangle):|q|=q_0\right\} \geq \frac{11 n \delta}{4}.
$$
By Bowen~\cite[Theorem~1.22]{Bow}, for every $q\in\R^d$ there exists a unique $\sigma^n$-invariant measure $\mu_q$ such that
$$
P_n(\langle s'(s)+q,\Phi_n\rangle-\langle q,n\vec{\alpha}\rangle)=h_{\mu_q}(\sigma^n)+\int\langle s'(s)+q,\Phi_n\rangle-\langle q,n\vec{\alpha}\rangle\mathrm{d}\mu_q.
$$
Since $\Phi_n$ is locally constant and therefore Hölder continuous, the function\linebreak $q \mapsto P_n(\langle s'(s)+q,\Phi_n\rangle-\langle q,n\vec{\alpha}\rangle)$ is analytic and convex. Moreover, for any $q_* \in \mathbb{R}^d$, we have
$$
\left.\nabla\right|_{q=q_*} P_n(\langle s'(s)+q,\Phi_n\rangle-\langle q,n\vec{\alpha}\rangle)=\int\left(\Phi_n-n \vec{\alpha}\right) \mathrm{d} \mu_{q_*}
$$
where $\mu_{q_*}$ is the equilibrium state for $\langle s'(s)+q,\Phi_n\rangle-\langle q,n\vec{\alpha}\rangle$. Note that the set
$$
Q=\left\{q:  P_n(\langle s'(s)+q,\Phi_n\rangle-\langle q,n\vec{\alpha}\rangle) \leq 2 n \delta\right\} \subset B\left(0, q_0\right)
$$
is convex. By convexity and \eqref{bound for the pressure}, we get
$$
\begin{aligned}
& \left|\nabla P_n(\langle s'(s)+q,\Phi_n\rangle-\langle q,n\vec{\alpha}\rangle)\right|\\
&\qquad\geq \frac{P_n(\langle s'(s)+q,\Phi_n\rangle-\langle q,n\vec{\alpha}\rangle)-P_n(\langle s'(s)+\tilde{q},\Phi_n\rangle-\langle\tilde{q},n\vec{\alpha}\rangle)}{|q-\tilde{q}|} \geq \frac{3 n \delta}{8 q_0}
\end{aligned}
$$
for all $q \in \partial Q$. Define $f_1, f_2: Q \rightarrow \mathbb{R}^d$ by setting
$$
f_1(q)=\int\left(\Phi_n-n \vec{\alpha}\right) \mathrm{d} \mu_q=\nabla P_n(\langle s'(s)+q,\Phi_n\rangle-\langle q,n\vec{\alpha}\rangle)
$$
and
$$
f_2(q)=\int\left(S_n \Phi-n \vec{\alpha}\right) \mathrm{d} \mu_q.
$$
By \cite[Lemma 4.2]{BJKR}, $f_1(Q) \supset B\left(0, \frac{3 n \delta}{8 q_0}\right)$. Since
$$
\left\|f_1-f_2\right\|_{\infty} \leq \varepsilon n<\frac{3 n \delta}{8 q_0}
$$
we have
$$
\left|(1-t) f_1(q)+t f_2(q)\right| \geq \frac{3 \delta n}{8 q_0}-t\left|f_1(q)-f_2(q)\right| \geq \frac{3 \delta n}{8 q_0}-\varepsilon n>0
$$
for all $t \in[0,1]$ and $q \in \partial Q$. It follows that $\left.f_1\right|_{\partial Q}$ and $\left.f_2\right|_{\partial Q}$ are homotopic on $\mathbb{R}^d \backslash\{0\}$ and hence $0 \in f_2(Q)$. This means that there exists $q_1 \in Q$ such that $\int S_n \Phi \mathrm{~d} \mu_{q_1}=n \vec{\alpha}$. Now, by \eqref{bound for the pressure},
$$
\begin{aligned}
\frac{3 n \delta}{4} & \leq P_n(\langle s'(s)+q_1,\Phi_n\rangle-\langle q_1,n\vec{\alpha}\rangle)\\
&=h_{\mu_{q_1}}( \sigma^n)+\int \langle s'(s)+q_1,\Phi_n\rangle-\langle q_1,n\vec{\alpha}\rangle \mathrm{d} \mu_{q_1}\\
&=h_{\mu_{q_1}}( \sigma^n)+\int \langle s'(s),S_n\Phi\rangle \mathrm{d} \mu_{q_1}+\frac{n\delta}{4}\\
&=h_{\mu_{q_1}}( \sigma^n)+\langle s'(s),n\vec{\chi}(\tilde{\mu}_{q_1}, \A)\rangle+\frac{n\delta}{4},\\
\end{aligned}
$$
where $\tilde{\mu}_{q_1}=\frac{1}{n}(S_n)_*\mu_{q_1}$ is $\sigma$-invariant $n$-step Bernoulli measure. Hence,
$$
0\geq\frac{\delta}{2\chi_{\lfloor s\rfloor+1}(\tilde{\mu}_{q_1},\A)}\geq (s-\lfloor s\rfloor)+\frac{h_{\tilde{\mu}_{q_1}}(\sigma)+\sum_{k=1}^{\lfloor s\rfloor}\chi_k(\tilde{\mu}_{q_1},\A)}{\chi_{\lfloor s\rfloor+1}(\tilde{\mu}_{q_1},\A)}=s-\dim_{\mathrm L}\tilde{\mu}_{q_1},
$$
which completes the proof.
\end{proof}

\begin{thm}\label{thm:usefulequations}
Let $\A:=\left(A_1, \ldots, A_N\right)$ be a finite tuple of matrices $A_i\in \glr$. Suppose that $(A_1, \ldots, A_N)$ is dominated. Then
$$
\begin{aligned}
 & \sup\left \{s \geq 0: \inf _{q \in \mathbb{R}^{d}}\left\{P\left(\log \psi^{s'(s)+q}(\A)\right)- \langle q ,\vec{\alpha} \rangle \right\} \geq 0\right\}\\
&\qquad=\sup \left\{\operatorname{dim}_{\mathrm{L}}\mu : \mu \in \M(\Sigma, \sigma) \text { and } \vec{\chi}(\mu)=\vec{\alpha}\right\} \\
&\qquad=\sup \left\{\operatorname{dim}_{\mathrm{L}}\mu : \mu \in \Ek_\sigma(\Sigma) \text { and } \vec{\chi}(\mu)=\vec{\alpha}\right\} \\
&\qquad=\sup \left\{\operatorname{dim}_{\mathrm{L}}\mu : \mu \text{ is a $n$-step Bernoulli measure } \text { and } \vec{\chi}(\mu)=\vec{\alpha}\right\} \\
&\qquad=\min _{k \in\{1, \ldots, d\}}\left\{k-1+\frac{h_{\text{top}}(E(\vec{\alpha}))-\sum_{i=1}^{k-1} \alpha_i}{\alpha_k}\right\}.
\end{aligned}
$$
\end{thm}

\begin{proof}
By Proposition \ref{existence of Bernoulli measure} and Lemma \ref{psi is almost additive},  we have
$$
\begin{aligned}
&\sup\left \{s \geq 0: \inf _{q \in \mathbb{R}^{d}}\left\{P\left(\log \psi^{s'(s)+q}(\A)\right)- \langle q ,\vec{\alpha} \rangle \right\} \geq 0\right\}\\
&\qquad\leq \sup \left\{\operatorname{dim}_{\mathrm{L}}\mu : \mu \text{ is a $n$-step Bernoulli measure } \text { and } \vec{\chi}(\mu)=\vec{\alpha}\right\} \\
&\qquad\leq \sup \left\{\operatorname{dim}_{\mathrm{L}}\mu : \mu \in \Ek_\sigma(\Sigma) \text { and } \vec{\chi}(\mu)=\vec{\alpha}\right\}\\
&\qquad\leq \sup \left\{\operatorname{dim}_{\mathrm{L}}\mu : \mu \in \M(\Sigma, \sigma) \text { and } \vec{\chi}(\mu)=\vec{\alpha}\right\}.\\
\end{aligned}
$$

Since, for every $\mu \in \M(\Sigma, \sigma)$ with $\vec{\chi}(\mu)=\vec{\alpha}$, we clearly have
$$
\begin{aligned}
&\operatorname{dim}_{\mathrm{L}}(\mu) \leq  \min _{k \in\{1, \ldots, d\}}\left\{k-1+\frac{h_{\text{top}}(E(\vec{\alpha}))-\sum_{i=1}^{k-1} \alpha_i}{\alpha_k}\right\}.
\end{aligned}
$$

 Write \[\Sigma_n(\vec{\alpha}, \varepsilon)=\left\{i \in \Sigma_{n} :  \left|\frac{1}{n}\log \sigma_{i}(\mathcal{A}_{i})-\alpha_i\right|<\ep  \text{ for }i=1, \ldots, d \right\}.\] Observe that
$$
\begin{aligned}
P\left(\log \psi^{s^{\prime}(s)-q}(\A)\right) & \geq \liminf _{n \rightarrow \infty} \frac{1}{n} \log \sum_{\mathrm{i} \in \Sigma_n(\vec{\alpha}, \varepsilon)} \psi^{s^{\prime}(s)+q}(\A_i) \\
& \geq-\left\langle s^{\prime}(s), \vec{\alpha}\right\rangle-C\left(\left|s^{\prime}(s)\right|+|q|\right) \varepsilon+\liminf _{n \rightarrow \infty} \frac{1}{n} \log \# \Sigma_n(\vec{\alpha}, \varepsilon)
\end{aligned}
$$
for all $s'(s)$, $q \in \R^d$. Recalling the definition of the topological entropy (see \cite{Bowen}), as $\varepsilon > 0$ is arbitrary, we get that

$$
\inf _{q \in \mathbb{R}^3} P\left(\log \psi^{s^{\prime}(s)-q}(\A)\right) \geq  h_{\text{top}} \Big(\sigma, E(\vec \alpha)\Big)-\left\langle s^{\prime}(s), \vec{\alpha}\right\rangle.
$$

Thus,
$$
\begin{aligned}
&\min _{k \in\{1, \ldots, d\}}\left\{k-1+\frac{h_{\text{top}}(E(\vec{\alpha}))-\sum_{i=1}^{k-1} \alpha_i}{\alpha_k}\right\} \\
&\qquad \leq \sup \left\{s \geq 0: \inf _{q \in \mathbb{R}^3}\left\{P\left(\log \psi^{s^{\prime}(s)-q}(\A)\right)-\langle q, \vec{\alpha}\rangle\right\} \geq 0\right\}.
\end{aligned}
$$
\end{proof}

We remind the reader that the (lower) Hausdorff dimension of the measure $\mu$ on $\mathbb{R}^2$ is defined by
$$
\operatorname{dim}_{\mathrm{H}}(\mu)=\inf \left\{\operatorname{dim}_{\mathrm{H}}(A): \mu(A)>0\right\} .
$$

To provide the lower bounds in Theorem \ref{main result for dominated case}, we find invariant measures with prescribed integrals and Lyapunov exponents, for which we can compute the Hausdorff dimension. The following theorem due to Rapaport \cite{Rapaport2024SelfAffine} guarantees that $n$-step Bernoulli measures can be used for this purpose.
\begin{thm}\label{Rapaport}
 Let $\left(A_1+v_1, \ldots, A_N+v_N\right)$ be an affine IFS on $\mathbb{R}^3$ satisfying the SOSC. If  $\left(A_1, \ldots, A_N\right) \in \gl3^N$ is typical, then
$$
\operatorname{dim}_{\mathrm{H}}\left(\pi_* \mu\right)=\operatorname{dim}_{\mathrm{L}}(\mu)
$$
for all Bernoulli measures $\mu$ on $\Sigma$.
    \end{thm}
\begin{proof}
    It follows from \cite[Theorem 1.9]{Rapaport2024SelfAffine}.
\end{proof}
We extend this result to higher-dimensional self-affine sets in Theorem \ref{thm:everyergodic} under the transversality assumption. Combining Theorem~\ref{thm:usefulequations} and Theorem~\ref{Rapaport}, we can have the following immediate result.

\begin{thm}\label{main result for dominated case}
Let $\Theta:=\left(A_1+v_i, \ldots, A_N+v_N\right)$ be an affine IFS on $\mathbb{R}^3$ satisfying SOSC, where $\{A_1, \ldots, A_{N}\} \in \gl3$. Suppose that $(A_1, \ldots, A_N)$ is dominated and typical. Then
$$
\begin{aligned}
\dim_{\mathrm{H}}\left(\pi E(\vec{\alpha})\right) & =\sup\left \{s \geq 0: \inf _{q \in \mathbb{R}^{3}}\left\{P\left(\log \psi^{s'(s)+q}(\A)\right)- \langle q ,\vec{\alpha} \rangle \right\} \geq 0\right\}\\
&=\sup \left\{\operatorname{dim}_{\mathrm{L}}\mu : \mu \in \M(\Sigma, \sigma) \text { and } \vec{\chi}(\mu)=\vec{\alpha}\right\} \\
&=\sup \left\{\operatorname{dim}_{\mathrm{L}}\mu : \mu \text{ is an $n$-step Bernoulli measure } \text { and } \vec{\chi}(\mu)=\vec{\alpha}\right\} \\
& =\min _{k \in\{1, 2, 3\}}\left\{k-1+\frac{h_{\text{top}}(E(\vec{\alpha}))-\sum_{i=1}^{k-1} \alpha_i}{\alpha_k}\right\}
\end{aligned}
$$
for $\vec{\alpha} \in \mathring{L}.$
\end{thm}

\begin{proof}
Recall that the limit in the definition of the topological pressure exists by Lemma~\ref{psi is almost additive}. Applying Theorem \ref{HD-one side} for $d=3$, 
\begin{equation}\label{1:eq}
\dim_{H}(\pi(E(\vec{\alpha}))\leq \left\{s \geq 0: \inf _{q \in \mathbb{R}^{3}}\left\{P\left(\log \psi^{s'(s)+q}(\A)\right)- \langle q ,\vec{\alpha} \rangle \right\} \geq 0\right\}.
\end{equation}

By Theorem \ref{Rapaport}, 
$$
\begin{aligned}
&\sup \left\{\operatorname{dim}_{\mathrm{L}}\mu : \mu \text{ is an $n$-step Bernoulli measure } \text { and } \vec{\chi}(\mu)=\vec{\alpha}\right\} \\
&\qquad=\sup \left\{\operatorname{dim}_{\mathrm{H}}\pi_*\mu : \mu \text{ is an $n$-step Bernoulli measure } \text { and } \vec{\chi}(\mu)=\vec{\alpha}\right\}\\
&\qquad\leq \operatorname{dim}_{\mathrm{H}}\left(\pi E(\vec{\alpha})\right).
\end{aligned}$$
Then the claim follows by Theorem~\ref{thm:usefulequations} applied for $d=3$.
\end{proof}

\section{Proof of Theorem \ref{t:main-result1}}\label{proof-main-result1}
We need the following Lemma to finish the proof of Theorem \ref{t:main-result1}.
\begin{lem}\label{lem:VP bound}
Let $\A:=\left(A_1, \ldots, A_N\right)$ be a finite tuple of matrices $A_i\in \glr$. Suppose that $(A_1, \ldots, A_N)$  is typical. Then, 
\[
\begin{aligned}
&\min_{k \in{1, \ldots, d}}
\left\{
k-1+
\frac{h_{\mathrm{top}}(E(\vec{\alpha}))
-\sum_{i=1}^{k-1} \alpha_i}{\alpha_k}
\right\} \\
&\qquad \leq
\sup \left\{
s \geq 0:
\inf _{q \in \mathbb{R}^3}
\left\{
P\left(\log \psi^{s^{\prime}(s)-q}(\A)\right)
-\langle q, \vec{\alpha}\rangle
\right\}
\geq 0
\right\}.
\end{aligned}
\]
\end{lem}

\begin{proof}
First, we note that the limit in the definition of the topological pressure exists by Theorem~\ref{vp-for the generalized sing}. Write
$$
\Sigma_n(\vec{\alpha}, \varepsilon)=\left\{i \in \Sigma_{n} :  \left\|\frac{1}{n}\Psi(\A_{i})-\vec{\alpha}\right\|<\ep\right\}.
$$

Observe that

$$
\begin{aligned}
P\left(\log \psi^{s^{\prime}(s)-q}(\A)\right) & \geq \liminf _{n \rightarrow \infty} \frac{1}{n} \log \sum_{\mathrm{i} \in \Sigma_n(\vec{\alpha}, \varepsilon)} \psi^{s^{\prime}(s)+q}(\A_i) \\
& \geq-\left\langle s^{\prime}(s), \vec{\alpha}\right\rangle-C\left(\left|s^{\prime}(s)\right|+|q|\right) \varepsilon+\liminf _{n \rightarrow \infty} \frac{1}{n} \log \# \Sigma_n(\vec{\alpha}, \varepsilon)
\end{aligned}
$$
for all $s'(s)$, $q \in \R^d$. By the definition of the topological entropy (see \cite{Bowen}), and since $\varepsilon>0$ is arbitrary, it follows that

$$
\inf _{q \in \mathbb{R}^d} P\left(\log \psi^{s^{\prime}(s)-q}(\A)\right) \geq  h_{\text{top}} \Big(\sigma, E(\vec \alpha)\Big)-\left\langle s^{\prime}(s), \vec{\alpha}\right\rangle.
$$

Consequently,
\[
\begin{aligned}
&\min_{k \in{1, \ldots, d}}
\left\{
k-1+
\frac{h_{\text{top}}(E(\vec{\alpha}))
-\sum_{i=1}^{k-1} \alpha_i}{\alpha_k}
\right\} \\
&\qquad \leq
\sup \left\{
s \geq 0:
\inf _{q \in \mathbb{R}^3}
\left\{
P\left(\log \psi^{s^{\prime}(s)-q}(\A)\right)
-\langle q, \vec{\alpha}\rangle
\right\}
\geq 0
\right\}.
\end{aligned}
\]
\end{proof}
\begin{proof}[Proof of Theorem \ref{t:main-result1}]

By Theorem \ref{main result for dominated case}, we have

$$
\begin{aligned}
&\dim_{\mathrm{H}}\left(\pi E^{n, \mathcal{D}}(n\vec{\alpha}) \right) \\
&\qquad\qquad=\sup\left \{s \geq 0: \inf _{q \in \mathbb{R}^{3}}\left\{P\left(\log \psi^{s'(s)+q}(\B)\right)- \langle q ,n\vec{\alpha} \rangle \right\} \geq 0\right\}\\
&\qquad\qquad=\sup \left\{\operatorname{dim}_{\mathrm{L}}\mu : \mu \in \M((\Sigma_{n}^{\D})^{\N}, f) \text { and } \vec{\chi}(\mu, \B)=n\vec{\alpha}\right\} \\
&\qquad\qquad=\sup \left\{\operatorname{dim}_{\mathrm{L}}\mu : \mu \text { is a Bernoulli measure on }(\Sigma_{n}^{\D})^{\N} \text { and } \vec{\chi}(\mu)=n\vec{\alpha}\right\} \\
&\qquad\qquad=\min _{k \in\{1, 2, 3\}}\left\{k-1+\frac{h_{\text{top}}(E (f, E^{n, \mathcal{D}}( n \vec{\alpha})) -n\sum_{i=1}^{k-1}\alpha_i}{n\alpha_k}\right\}.\\
\end{aligned}
$$
Using Lemma \ref{lem: the relation between topological entropy} and Lemma~\ref{lem:VP bound}, we have
$$\begin{aligned}
\lim_{n\to\infty}\dim_{\mathrm{H}}\left(\pi E^{n, \mathcal{D}}(n\vec{\alpha}) \right)
&\leq 
 \min_{k\in\{1,2,3\}}
 \left\{
 k-1+
 \frac{
 h_{\mathrm{top}}\!\left(\sigma,E(\vec\alpha)\right)
 -\sum_{i=1}^{k-1}\alpha_i
 }{\alpha_k}
 \right\}\\
 &\leq\sup\left \{s \geq 0: \inf _{q \in \mathbb{R}^{d}}\left\{P\left(\log \psi^{s'(s)+q}(\A)\right)- \langle q ,\vec{\alpha} \rangle \right\} \geq 0\right\}.
\end{aligned}$$
On the other hand, Theorem \ref{continuity_potential} gives
$$
\begin{aligned}
&\lim_{n\to\infty}\dim_{\mathrm{H}}\left(\pi E^{n, \mathcal{D}}(n\vec{\alpha}) \right)\\
&\qquad\qquad=\lim_{n\to\infty}\sup\left \{s \geq 0: \inf _{q \in \mathbb{R}^{d}}\left\{P\left(\log \psi^{s'(s)+q}(\B)\right)- \langle q ,n\vec{\alpha} \rangle \right\} \geq 0\right\}\\
&\qquad\qquad=\sup\left \{s \geq 0: \inf _{q \in \mathbb{R}^{d}}\left\{P\left(\log \psi^{s'(s)+q}(\A)\right)- \langle q ,\vec{\alpha} \rangle \right\} \geq 0\right\}.
\end{aligned}
$$
Finally, by Theorem~\ref{HD-one side}, we have
$$
\begin{aligned}
&\sup\left \{s \geq 0: \inf _{q \in \mathbb{R}^{d}}\left\{P\left(\log \psi^{s'(s)+q}(\A)\right)- \langle q ,\vec{\alpha} \rangle \right\} \geq 0\right\}\\
&\qquad=\lim_{n\to\infty}\dim_{\mathrm{H}}\left(\pi E^{n, \mathcal{D}}(n\vec{\alpha}) \right)\\
&\qquad\leq\dim_{\mathrm{H}}\pi E(\vec{\alpha})\\
&\qquad\leq\sup\left \{s \geq 0: \inf _{q \in \mathbb{R}^{d}}\left\{P\left(\log \psi^{s'(s)+q}(\A)\right)- \langle q ,\vec{\alpha} \rangle \right\} \geq 0\right\}.
\end{aligned}
$$
Finally, for any $\mu' \in \M((\Sigma_{n}^{\D})^{\N} , f)$, Proposition \ref{relation between entropies and LE} provides a measure $\mu \in \M(\Si, \sigma)$ such that $\diml \mu=\diml \mu'$, which completes the proof. 
\end{proof}

\section{Simultaneous Transversality}\label{sec: simultaneous-transversality}

In this section, we use the transversality method to tackle higher dimensions instead of Rapaport's result \cite{Rapaport2024SelfAffine}, which is well-suited up to dimension three. First, let us introduce the transversality condition. Consider smooth functions $t_i\colon\R^M\to\R^d$ for $i=1,\ldots,N$, and a parametrised family of affine IFSs $\Theta_v=\left(A_1 + t_1(v), \ldots, A_N + t_N(v)\right)$. We say that $\Theta_v$ satisfies the \emph{transversality condition} if for every open and bounded set $U\subset\R^M$ there exists $C>0$ such that for every $\mathrm i\neq\mathrm j\in\Sigma$
\begin{equation}\label{eq:transineq}
\L\left(\left\{v\in U:\|\pi_v(\mathrm i)-\pi_v(\mathrm j)\|<r\right\}\right)\leq C\prod_{k=1}^d\min\left\{1,\frac{r}{\alpha_k(\A_{\mathrm i\wedge\mathrm j})}\right\},
\end{equation}
where $\mathcal{L}$ denotes the Lebesgue measure of $\R^M$. The main statement of this section is the following.

\begin{thm}\label{thm:everyergodic}
   Let $\Theta_v=\left(A_1 + t_1(v), \ldots, A_N + t_N(v)\right)$ be a parametrised family of affine IFS with canonical projection $\pi_v$ such that it satisfies \eqref{eq:transineq}. Then, for Lebesgue-almost all $v$
    $$
    \dim_{\mathrm{H}}(\pi_v)_*\mu=\min\{d,\dim_{\mathrm{L}}\mu\}\text{  for every $\mu\in\Ek_\sigma(\Sigma)$}.
    $$
\end{thm}

To prove Theorem~\ref{thm:everyergodic}, we follow the method of \cite{BSS25}. Let us first introduce a family of ultrametrics generating the Borel $\sigma$-algebra generated by cylinder sets. We say that the metric $\rho$ on $\Sigma$ is an \emph{exponential metric} if
\begin{enumerate}[(i)]
	\item\label{it: rho to psi} there exists a function $\psi : \Sigma_* \to (0,\infty)$ so that $\rho(\mathrm{i}, \mathrm{j}) = \begin{cases}\psi(\mathrm{i} \wedge \mathrm{j})&\text{ if }\mathrm{i} \neq \mathrm{j} \in \Sigma,\\0&\text{otherwise};\end{cases}$
	\item\label{it: psi mono} there exists $\gamma \in (0,1)$ such that $\psi(\mathrm{i}|_{n+1}) \leq \gamma \psi(\mathrm{i}|_{n})$ for each $n \geq 1$ and $\mathrm{i} \in \Sigma$.
\end{enumerate}
The metric 
\begin{equation}\label{eq:rhok}
\rho_k(\mathrm{i},\mathrm j):=\alpha_k(\A_{\mathrm i\wedge\mathrm j})
\end{equation}
serves as a natural example of an exponential metric in our setting for every $k=1,\ldots,d$. Now, we introduce the concept of relative entropy from \cite{BSS25}. Let $\mu$ and $\nu$ be finite Borel measures on the metric space $(\Sigma,\rho)$, where $\rho$ is an exponential ultrametric. The \emph{relative dimension} of $\mu$ with respect to $\nu$ is
\[
    \dim(\mu || \nu, \rho) := \inf \left\{ \eps > 0 : - \eps < \underset{\mathrm{i} \sim \mu}{\essinf} \liminf \limits_{r \to 0} \frac{\log \frac{\mu(B_\rho(\mathrm{i},r))}{\nu(B_\rho(\mathrm{i},r))}}{\log r} \leq  \underset{\mathrm{i} \sim \mu}{\esssup} \limsup \limits_{r \to 0} \frac{\log \frac{\mu(B_\rho(\mathrm{i},r))}{\nu(B_\rho(\mathrm{i},r))}}{\log r} < \eps  \right\},
\]
where $B_\rho(\mathrm i, r)$ denotes the open $r$-ball centred at $\mathrm i$ in metric $\rho$. We adopt the convention that $\log \frac{\mu(B_\rho(\mathrm i,r))}{0} = + \infty$ if $\mu(B_\rho(\mathrm i,r)) > 0$. In particular $\dim(\mu || \nu, \rho) = \infty$ if $\mu(\Sigma \setminus \supp(\nu)) > 0$. 

In the case of metric $\rho_k$ defined in \eqref{eq:rhok}, the relative dimension can be written as 
\small{\begin{equation}\label{eq:reldim}
    \dim(\mu || \nu, \rho)=\inf \left\{ \eps > 0 : - \eps < \underset{\mathrm{i} \sim \mu}{\essinf} \liminf \limits_{n\to\infty} \dfrac{\log \frac{\mu([\mathrm{i}|_n])}{\nu([\mathrm{i}|_n])}}{\log \alpha_k(\A_{\mathrm{i}|_n})} \leq  \underset{\mathrm{i} \sim \mu}{\esssup} \limsup \limits_{n\to\infty} \dfrac{\log \frac{\mu([\mathrm{i}|_n])}{\nu([\mathrm{i}|_n])}}{\log \alpha_k(\A_{\mathrm{i}|_n})} < \eps \right\},
\end{equation}}
Note that the relative dimension is not a metric on the space of measures, since it is not symmetric; however, we might think of it as a semi-metric. So, we introduce the notation:
$$
\widetilde{B}_\rho(\nu,\delta):=\{\mu\in\Ek_\sigma(\Sigma):\dim(\mu || \nu,\rho)<\delta\}.
$$
In the special case of metric $\rho_k$, we will denote this ball by $\widetilde{B}_k$ for simplicity.

We say that a set $\mathcal P \subset \Mk(\Sigma)$ of Borel probability measures is \emph{relative dimension separable} with respect to the metric $\rho$ if there exists a countable set $\Vk \subset \mathcal P$ such that for every $\mu \in \mathcal P$ and $\eps > 0$ there exists $\nu \in \Vk$ with $\dim(\mu || \nu, \rho) < \eps$.

\begin{prop}\label{prop: general ergodic rds}
Let $\rho$ be a metric on $\Sigma$ satisfying \ref{it: rho to psi} and \ref{it: psi mono}. Then the set $\Ek_\sigma(\Sigma)$ of ergodic shift-invariant probability measures on $\Sigma$ is relative dimension separable with respect to $\rho$.
\end{prop}

\begin{proof}The proof can be found in \cite[Proposition~7.10]{BSS25}.\end{proof}

  Fix $\varepsilon>0$ and $\nu\in\Mk(\Sigma)$. For $\mu\in \widetilde{B}_k(\nu,\varepsilon)$ with $k-1<\dim_{\mathrm{L}}\mu\leq k$ and $M\in\N$, let 
    \begin{equation}\label{eq:defC}
        C_\mu(M,k,\nu)=\left\{\mathrm i:\alpha_k(\A_{\mathrm i|_n})^{\varepsilon}\leq\frac{\mu([\mathrm i|_n])}{\nu([\mathrm i|_n])}\leq\alpha_k(\A_{\mathrm i|_n})^{-\varepsilon}\text{ for every }n\geq M\right\},
    \end{equation}
    and
    \begin{equation}\label{eq:defD}
        D_\mu(M,k)=\left\{\mathrm i:\alpha_{k}(\A_{\mathrm i|_n})^{\dim_{\mathrm{L}}\mu+\varepsilon}\leq\frac{\mu([\mathrm i|_n])}{\varphi^{{k}-1}(\A_{\mathrm i|_n})}\leq\alpha_{k}(\A_{\mathrm i|_n})^{\dim_{\mathrm{L}}\mu-\varepsilon}\text{ for every }n\geq M\right\}.
    \end{equation}
    By definition, 
    $$D_\mu(M,k)\subseteq D_\mu(M+1,k)\text{ and } C_\mu(M,k,\nu)\subseteq C_\mu(M+1,k,\nu).$$ Furthermore, $\mu(\bigcup_{M=1}^\infty C_\mu(M,k))=\mu(\bigcup_{M=1}^\infty D_\mu(M,k))=1$. For every $\mu\in \widetilde{B}_k(\nu,\varepsilon)$ with $k-1<\dim_{\mathrm{L}}\mu\leq k$, let $M_\mu(\nu,\varepsilon)\in\N$ be the smallest integer such that 
    $$
    \mu(C_\mu(M_\mu(\nu,\varepsilon),k,\nu)\cap D_\mu(M_\mu(\nu,\varepsilon),k))>1-\varepsilon,
    $$
    and let 
    $$B(\mu,\varepsilon,\nu,k):=C_\mu(M_\mu(\nu,\varepsilon),k,\nu)\cap D_\mu(M_\mu(\nu,\varepsilon),k).$$

For $\nu\in\Ek_\sigma(\Sigma)$, $M\in\N$ and $t,\varepsilon>0$, let 
\begin{equation*}
V(\nu,\varepsilon,M,t):=\left\{\mu\in\widetilde{B}_k(\nu,\varepsilon):\max\{k-1,t-\varepsilon\}<\dim_{\mathrm{L}}\mu\leq\min\{k,t+\varepsilon\}\text{ and } M_\mu(\nu,\varepsilon)\leq M\right\}.
\end{equation*}

\begin{prop}\label{prop:energyintegral}
    Let $\Theta_v=\left(A_1 + t_1(v), \ldots, A_N + t_N(v)\right)$ be a parameterised family of affine IFS with canonical projection $\pi_v$ such that it satisfies \eqref{eq:transineq}. Then there exists $\delta>0$ such that for every open and bounded set $U$, every $\varepsilon>0$, every $t\in\R_+$, every $M\in\N$ and $\nu\in\Ek_\sigma(\Sigma)$,
    \[
    \begin{aligned}
    I &:=\int_U\sup_{\mu\in V(\nu,\varepsilon,M,t)}\iint\|\pi_v(\mathrm i)-\pi_v(\mathrm j)\|^{-(t-\delta\varepsilon)}d\mu(\mathrm i)d\mu|_{B(\mu,\varepsilon,\nu,k)}(\mathrm j)dv<\infty.
    \end{aligned}
    \]
\end{prop}

\begin{proof}
For every $\mu\in\widetilde{B}_k(\nu,\varepsilon)$ with $\max\{k-1,t-\varepsilon\}<\dim_{\mathrm{L}}\mu\leq\min\{k,t+\varepsilon\}$ and $M_\mu(\nu,\varepsilon)\leq M$,  we have
\begin{equation}\label{eq:ub0}
\mu([\mathrm i|_n])\leq\nu([\mathrm i|_n])\alpha_k(\A_{\mathrm i|_n})^{-\varepsilon}
\end{equation}
and
\begin{equation}\label{eq:ub}
\nu([\mathrm i|_n])\leq\mu([\mathrm i|_n])\alpha_k(\A_{\mathrm i|_n})^{-\varepsilon}\leq \varphi^{k-1}(\A_{\mathrm i|_n})\alpha_k(\A_{\mathrm i|_n})^{\dim_{\mathrm L}\mu-2\varepsilon}\leq\varphi^{k-1}(\A_{\mathrm i|_n})\alpha_k(\A_{\mathrm i|_n})^{t-3\varepsilon}
\end{equation}
 for every $\mathrm i\in B(\mu,\varepsilon,\nu,k)$ and for every $n\geq M$. Let 
\[
 E:=\{\mathrm i\in\Sigma:\nu([\mathrm i|_n])\leq\varphi^{k-1}(\A_{\mathrm i|_n})\alpha_k(\A_{\mathrm i|_n})^{t-3\varepsilon}=\varphi^{t-3\varepsilon}(\A_{\mathrm i|_n})\text{ for }n\geq M\},
\]
and so $B(\mu,\varepsilon,\nu,k)\subseteq E$.
We claim that $\delta$ can be chosen to be $(2d+6)\beta$, where $\beta=\frac{\log\min_i\|A_i^{-1}\|^{-1}}{\log\max_i\|A_i\|}$. For short, we write $s=t-(2d+6)\beta\varepsilon$, $V=V(\nu,\varepsilon,M,t)$ and $B=B(\mu,\varepsilon,\nu,k)$. Let $m(r):=\min\{m\geq 1:(\max_i\|A_i\|)^m<r^{-1/s}\}$. Then for every $\mathrm k,\mathrm i\in\Sigma_*$ with $|\mathrm i|=m(r)$
\begin{equation}\label{eq:needtech1}
    \alpha_k(\A_{\mathrm k\mathrm i})\geq\alpha_k(\A_{\mathrm k})(\max_i\|A_i\|)^{m(r)\beta}\geq \frac12\alpha_k(\A_{\mathrm k})r^{-\beta/s}.
\end{equation}
For $\mathrm i\in\Sigma_*$, let $\mathcal C_{\mathrm i}=\{(\mathrm j,\mathrm k)\in\Sigma\times\Sigma:\mathrm j\wedge\mathrm k=\mathrm i\}$. Now we have all the ingredients to estimate $I$. 
{\allowdisplaybreaks
\begin{eqnarray*}
   & I&\leq\sum_{\mathrm k\in\Sigma_*}\int_U\sup_{\mu\in V}\iint_{\mathcal C_{\mathrm k}}\|\pi_v(\mathrm i)-\pi_v(\mathrm j)\|^{-s}d\mu|_{B}(\mathrm i)d\mu|_{B}(\mathrm j)dv\\
   & &\leq1+\sum_{\mathrm k\in\Sigma_*}\int_U\sup_{\mu\in V}\int_1^\infty\mu\times\mu|_{B}(\{(\mathrm{i},\mathrm j)\in\C_{\mathrm k}:\|\pi_v(\mathrm i)-\pi_v(\mathrm j)\|<r^{-1/s}\})drdv\\
    & &\leq1+\sum_{\mathrm k\in\Sigma_*}\int_U\sup_{\mu\in V}\int_1^\infty\sum_{\substack{\mathrm j\in\Sigma_{m(r)}\\\mathrm i\in\Sigma_{m(r)}}}\mu([\mathrm k\mathrm j])\mu|_{B}([\mathrm k\mathrm i])\mathbb{I}(\|\A_{\mathrm k}(\pi_v(\mathrm i)-\pi_v(\mathrm j))\|<3r^{-1/s})drdv\\
&     &\leq1+\sum_{\mathrm k\in\Sigma_*}\int_U\sup_{\mu\in V}\int_1^\infty\sum_{\substack{\mathrm j\in\Sigma_{m(r)}\\\mathrm i\in\Sigma_{m(r)}}}\mu([\mathrm k\mathrm j])\mu([\mathrm k\mathrm i])\mathbb{I}([\mathrm k\mathrm i]\cap B\neq\emptyset)
    \\
& & \qquad\qquad\mathbb{I}(\|\A_{\mathrm k}(\pi_v(\mathrm i)-\pi_v(\mathrm j))\|<3r^{-1/s})drdv\\
\mbox{by \eqref{eq:ub0}}\\
&     &\leq C+\sum_{\substack{\mathrm k\in\Sigma_*\\|\mathrm{k}|\geq M}}\int_U\int_1^\infty\sum_{\substack{\mathrm j\in\Sigma_{m(r)}\\\mathrm i\in\Sigma_{m(r)}}}\nu([\mathrm k\mathrm j])\nu([\mathrm k\mathrm i])\alpha_k(\A_{\mathrm k\mathrm i})^{-\varepsilon}\alpha_k(\A_{\mathrm k\mathrm j})^{-\varepsilon}\\
& &\qquad\qquad \mathbb{I}([\mathrm k\mathrm i]\cap E\neq\emptyset)\mathbb{I}(\|\A_{\mathrm k}(\pi_v(\mathrm i)-\pi_v(\mathrm j))\|<3r^{-1/s})drdv\\
&     &=C+\sum_{\substack{\mathrm k\in\Sigma_*\\|\mathrm{k}|\geq M}}\int_1^\infty\sum_{\substack{\mathrm j\in\Sigma_{m(r)}\\\mathrm i\in\Sigma_{m(r)}}}\nu([\mathrm k\mathrm j])\nu([\mathrm k\mathrm i])\alpha_k(\A_{\mathrm k\mathrm i})^{-\varepsilon}\alpha_k(\A_{\mathrm k\mathrm j})^{-\varepsilon}
    \\
& &\qquad\qquad \mathbb{I}([\mathrm k\mathrm i]\cap E\neq\emptyset)\mathcal{L}(v\in U:\|\A_{\mathrm k}(\pi_v(\mathrm i)-\pi_v(\mathrm j))\|<3r^{-1/s})dr\\
\mbox{by \eqref{eq:transineq}}\\
 &     &=C+\sum_{\substack{\mathrm k\in\Sigma_*\\|\mathrm{k}|\geq M}}\int_1^\infty\sum_{\substack{\mathrm j\in\Sigma_{m(r)}\\\mathrm i\in\Sigma_{m(r)}}}\nu([\mathrm k\mathrm j])\nu([\mathrm k\mathrm i])\alpha_k(\A_{\mathrm k\mathrm i})^{-\varepsilon}\alpha_k(\A_{\mathrm k\mathrm j})^{-\varepsilon}
     \\
& &\qquad\qquad
\mathbb{I}([\mathrm k\mathrm i]\cap E\neq\emptyset)\prod_{\ell=1}^d\min\left\{1,\frac{r^{-1/s}}{\alpha_\ell(\A_{\mathrm k})}\right\}dr\\
\mbox{by \eqref{eq:needtech1}}\\
  &    &\leq C+2^{2\varepsilon}\sum_{\substack{\mathrm k\in\Sigma_*\\|\mathrm{k}|\geq M}}\nu([\mathrm k])^2\alpha_k(\A_{\mathrm k})^{-2\varepsilon}\mathbb{I}([\mathrm k]\cap E\neq\emptyset)\int_1^\infty r^{2\varepsilon\beta/s}\prod_{\ell=1}^d\min\left\{1,\frac{r^{-1/s}}{\alpha_\ell(\A_{\mathrm k})}\right\}dr\\
   &   &\leq C+2^{2\varepsilon}\sum_{\substack{\mathrm k\in\Sigma_*\\|\mathrm{k}|\geq M}}\nu([\mathrm k])^2\alpha_k(\A_{\mathrm k})^{-2\varepsilon}\mathbb{I}([\mathrm k]\cap E\neq\emptyset)(\varphi^{s/(1-2\beta\varepsilon)}(\A_{\mathrm k}))^{-1}\\
\mbox{by \eqref{eq:ub}}\\
    &  &\leq C+2^{2\varepsilon}\sum_{\substack{\mathrm k\in\Sigma_*\\|\mathrm k|\geq M}}\nu([\mathrm k])\varphi^{t-5\varepsilon}(\A_{\mathrm k})(\varphi^{s/(1-2\beta\varepsilon)}(\A_{\mathrm k}))^{-1}\\
    &  &\leq C+2^{2\varepsilon}\sum_{n=M}^\infty(\max_i\|A_i\|)^{n(t-5\beta\varepsilon-s/(1-2\beta\varepsilon))}.
\end{eqnarray*}

}
Choosing $s$ such that $(1-2\beta\varepsilon)(t-5\beta\varepsilon)>t-(2t+5)\beta\varepsilon\geq t-(2d+5)\beta\varepsilon>s$, we get that the series is convergent.
\end{proof}

\begin{proof}[Proof of Theorem~\ref{thm:everyergodic}]
    Let $U$ be an arbitrary open and bounded set. Then by Proposition~\ref{prop:energyintegral}, for every $\nu\in\M(\Sigma)$, $\varepsilon>0$, $t\in\R$ and $M\in\N$ there exists $\widetilde{U}_{\nu,\varepsilon,t,M}\subseteq U$ such that $\L(U\setminus \widetilde{U}_{\nu,\varepsilon,t,M})=0$ and for every $\mu\in V(\nu,\varepsilon,M,t)$ and $v\in\widetilde{U}_{\nu,\varepsilon,t,M}$
    \begin{equation}\label{eq:contra}
    \dim_{\mathrm H}(\pi_v)_*\mu\geq\min\{d,\dim_{\mathrm L}\mu\}-(\delta+1)\varepsilon,
    \end{equation}
    where $\delta>0$ is given in Proposition~\ref{prop:energyintegral}. 
    
    Let $\mathcal V$ be a countable, relative dimension-dense subset of $\Ek_\sigma(\Sigma)$ by Proposition~\ref{prop: general ergodic rds}. Hence,
    \begin{equation}\label{eq:coverergodicmeasures}
    \bigcap_{n=1}^\infty \bigcup_{t\in\mathbb Q}\bigcup_{M\in\N}\bigcup_{\nu\in\mathcal V}V(\nu,1/n,t,M)=\Ek_\sigma(\Sigma).
    \end{equation}
    Let
    $$
    \widetilde{U}=\bigcap_{n=1}^\infty\bigcap_{t\in\mathbb Q}\bigcap_{M\in\N}\bigcap_{\nu\in\mathcal V}\widetilde{U}_{\nu,1/n,t,M}.
    $$
    Clearly, $\mathcal{L}(U\setminus\widetilde{U})=0$. We claim that for every $v\in\widetilde{U}$ and every $\mu\in\Ek_\sigma(\Sigma)$, $\dim_{\mathrm H}(\pi_v)_*\mu=\min\{d,\dim_{\mathrm L}\mu\}$. Let us argue by contradiction. That is, there exists $v\in\widetilde{U}$ and $\mu\in\Ek_\sigma(\Sigma)$ such that $$\dim_{\mathrm H}(\pi_v)_*\mu<\min\{d,\dim_{\mathrm L}\mu\}.$$ Choose $n$ such that 
    $$
    \dim_{\mathrm H}(\pi_v)_*\mu+\frac{1+\delta}{n}<\min\{d,\dim_{\mathrm L}\mu\}.
    $$
    By \eqref{eq:coverergodicmeasures}, there exists $t\in\mathbb Q$, $M\in\N$ and $\nu\in\mathcal V$ such that $\mu\in V(\nu,1/n,t,M)$. But since $v\in \widetilde{U}_{\nu,1/n,t,M}$, this contradicts to \eqref{eq:contra}.
\end{proof}

\section{Proof of Theorem \ref{t:main-result2}}\label{proof-main-result2}
We need the following theorem to finish the proof of Theorem \ref{t:main-result2}. 
\begin{thm}\label{thm: themB for the dominated case}
 Let $\Theta_v=\left(A_1 + t_1(v), \ldots, A_N + t_N(v)\right)$ be a parametrised family of affine IFS with canonical projection $\pi_v$ such that it satisfies \eqref{eq:transineq} and  $\left(A_1, \ldots, A_N\right)$ is dominated. Then, for Lebesgue-almost all $v$, we have $$
\begin{aligned}
\dim_{\mathrm{H}}\left(\pi_{v} E(\vec{\alpha})\right) & =\sup\left \{s \geq 0: \inf _{q \in \mathbb{R}^{3}}\left\{P\left(\log \psi^{s'(s)+q}(\A)\right)- \langle q ,\vec{\alpha} \rangle \right\} \geq 0\right\}\\
&=\sup \left\{\operatorname{dim}_{\mathrm{L}}\mu : \mu \in \M(\Sigma, \sigma) \text { and } \vec{\chi}(\mu)=\vec{\alpha}\right\} \\
&=\sup \left\{\operatorname{dim}_{\mathrm{L}}\mu : \mu \in \Ek_\sigma(\Sigma) \text { and } \vec{\chi}(\mu)=\vec{\alpha}\right\} \\
& =\min _{k \in\{1, \ldots, d\}}\left\{k-1+\frac{h_{\mathrm{top}}(E(\vec{\alpha}))-\sum_{i=1}^{k-1} \alpha_i}{\alpha_k}\right\}
\end{aligned}
$$
for all $\vec{\alpha} \in \mathring{\vec{L}}.$
\end{thm}

\begin{proof}
By Theorem \ref{HD-one side}, for every $v\in\R^{dN}$
\begin{equation}\label{1:eq}
\dim_{H}(\pi_{v}(E(\vec{\alpha}))\leq \left\{s \geq 0: \inf _{q \in \mathbb{R}^{d}}\left\{P\left(\log \psi^{s'(s)+q}(\A)\right)- \langle q ,\vec{\alpha} \rangle \right\} \geq 0\right\}.
\end{equation}
By Theorem ~\ref{thm:everyergodic}, there exists a set $U\subset\R^M$ such that $\mathcal{L}(\R^M\setminus U)=0$ and for every $v\in U$
$$
\begin{aligned}
\sup \left\{\operatorname{dim}_{\mathrm{L}}\mu : \mu \in \Ek_\sigma(\Sigma) \text { and } \vec{\chi}(\mu)=\vec{\alpha}\right\}&=\sup \left\{\operatorname{dim}_{\mathrm{H}}\pi_*\mu : \mu \in \Ek_\sigma(\Sigma) \text { and } \vec{\chi}(\mu)=\vec{\alpha}\right\}\\
&\leq \operatorname{dim}_{\mathrm{H}}\left(\pi_{v} E(\vec{\alpha})\right)
\end{aligned}$$
for every $\vec{\alpha}\in\mathring{\vec{L}}$. Then the claim follows by Theorem \ref{thm:usefulequations}.
\end{proof}

Let us recall the transversality lemma from \cite[Proposition~10.4.1]{BSSbook}. 

\begin{lem}\label{lem:transversality}
    Let $\{A_1, \ldots, A_N\} \subset \glr$ be such that $\max_{i\neq j}\|A_i\|+\|A_j\|< 1$. Then the IFS $\Theta_v=(A_1+v_1,\ldots,A_N+v_N)$ satisfies the transversality condition \eqref{eq:transineq} with parameters $v=(v_i)_{i=1}^N\in\R^{dN}$. 
\end{lem}

Now, we finish the paper by proving Theorem~\ref{t:main-result2}. 

\begin{proof}[Proof of Theorem \ref{t:main-result2}]
Let $\{A_1, \ldots, A_N\} \subset \glr$ be such that $\max_{i\neq j}\|A_i\|+\|A_j\|< 1$, and  $\left(A_1, \ldots, A_N\right)$ is typical, and let $\Theta_v=\left(A_1 + v_1, \ldots, A_N + v_N\right)$ be a parametrised family of affine IFS in $\R^d$ with parameters $v=(v_i)_{i=1}^N$. Then by Lemma~\ref{lem:transversality}, $\Theta_v$ satisfies the transversality assumption. In particular, the parametrised families of the IFSs $\Theta_v^{(n)}=(\A_{I}+t_I(v))_{I\in\Sigma_n^{\mathcal{D}}}$ satisfy the transversality assumption for every $n\geq2K_0$, where $t_I(v)=\sum_{n=1}^{|I|} \A_{\left.\mathrm{i}\right|_{n-1}} v_{i_n}$.

By Theorem \ref{thm: themB for the dominated case},  we have that there exists $U\subseteq\R^{dN}$ such that $\mathcal{L}(\R^{dN}\setminus U)=0$ and for every $v\in U$
$$
\begin{aligned}
\dim_{\mathrm{H}}\left(\pi_{v} E^{n, \mathcal{D}}( n \vec{\alpha}) \right) & =\sup\left \{s \geq 0: \inf _{q \in \mathbb{R}^{d}}\left\{P\left(\log \psi^{s'(s)+q}(\B)\right)- \langle q ,n\vec{\alpha} \rangle \right\} \geq 0\right\}\\
&=\sup \left\{\operatorname{dim}_{\mathrm{L}}\mu : \mu \in \M((\Sigma_{n}^{\D})^{\N}, f) \text { and } \vec{\chi}(\mu, \B)=n\vec{\alpha}\right\} \\
&=\sup \left\{\operatorname{dim}_{\mathrm{L}}\mu : \mu \text { is a Bernoulli measure on }(\Sigma_{n}^{\D})^{\N} \text { and } \vec{\chi}(\mu)=n\vec{\alpha}\right\} \\
& =\min _{k \in\{1, \ldots, d\}}\left\{k-1+\frac{h_{\text{top}}(E (f, E^{n, \mathcal{D}}( n \vec{\alpha})) -n\sum_{i=1}^{k-1}\alpha_i}{n\alpha_k}\right\}
\end{aligned}
$$
for every $\vec{\alpha}\in\mathring{\vec{L}}$ and for every $n\geq2K_0$. Using Lemma \ref{lem: the relation between topological entropy} and Lemma~\ref{lem:VP bound}, we have
$$\begin{aligned}
\lim_{n\to\infty}\dim_{\mathrm{H}}\left(\pi_v E^{n, \mathcal{D}}( n \vec{\alpha}) \right)
&\leq 
 \min_{k \in\{1, \ldots, d\}}
 \left\{
 k-1+
 \frac{
 h_{\mathrm{top}}\!\left(\sigma,E(\vec\alpha)\right)
 -\sum_{i=1}^{k-1}\alpha_i
 }{\alpha_k}
 \right\}\\
 &\leq\sup\left \{s \geq 0: \inf _{q \in \mathbb{R}^{d}}\left\{P\left(\log \psi^{s'(s)+q}(\A)\right)- \langle q ,\vec{\alpha} \rangle \right\} \geq 0\right\}
\end{aligned}$$
for every $\vec{\alpha}\in\mathring{\vec{L}}$ and $v\in U$. On the other hand, Theorem \ref{continuity_potential} gives
$$
\begin{aligned}
&\lim_{n\to\infty}\dim_{\mathrm{H}}\left(\pi_v E^{n, \mathcal{D}}( n \vec{\alpha}) \right)\\
&\qquad\qquad=\lim_{n\to\infty}\sup\left \{s \geq 0: \inf _{q \in \mathbb{R}^{d}}\left\{P\left(\log \psi^{s'(s)+q}(\B)\right)- \langle q ,n\vec{\alpha} \rangle \right\} \geq 0\right\}\\
&\qquad\qquad=\sup\left \{s \geq 0: \inf _{q \in \mathbb{R}^{d}}\left\{P\left(\log \psi^{s'(s)+q}(\A)\right)- \langle q ,\vec{\alpha} \rangle \right\} \geq 0\right\}.
\end{aligned}
$$
Finally, by Theorem~\ref{HD-one side}, we have
$$
\begin{aligned}
&\sup\left \{s \geq 0: \inf _{q \in \mathbb{R}^{d}}\left\{P\left(\log \psi^{s'(s)+q}(\A)\right)- \langle q ,\vec{\alpha} \rangle \right\} \geq 0\right\}\\
&\qquad=\lim_{n\to\infty}\dim_{\mathrm{H}}\left(\pi_v E^{n, \mathcal{D}}( n \vec{\alpha}) \right)\\
&\qquad\leq\dim_{\mathrm{H}}\pi_v E(\vec{\alpha})\\
&\qquad\leq\sup\left \{s \geq 0: \inf _{q \in \mathbb{R}^{d}}\left\{P\left(\log \psi^{s'(s)+q}(\A)\right)- \langle q ,\vec{\alpha} \rangle \right\} \geq 0\right\}.
\end{aligned}
$$
for every $\vec{\alpha}\in\mathring{\vec{L}}$ and $v\in U$. Finally, for any $\mu' \in \M((\Sigma_{n}^{\D})^{\N} , f)$, Proposition \ref{relation between entropies and LE} provides a measure $\mu \in \M(\Si, \sigma)$ such that $\diml \mu=\diml \mu'$, which completes the proof. 
\end{proof}

\subsection*{Acknowledgements.}
RM acknowledges the hospitality of Budapest University of Technology and Economics and the support of  the Knut and Alice
Wallenberg Foundation and the Swedish Research
Council grant 10465132. BB acknowledges the hospitality of Uppsala University and was supported by the grants NKFI FK134251, K142169, and the grant NKFI KKP144059
“Fractal geometry and applications”.

\bibliographystyle{acm}
\bibliography{Hausdorff.Spectrum}

\end{document}